\documentclass[12pt,a4paper,psamsfonts]{amsart}
\usepackage{amssymb,amscd,amsxtra,calc}
\usepackage{cmmib57}
\usepackage{graphicx}

\usepackage{lipsum}
\usepackage{amsfonts}
\usepackage{graphicx}
\usepackage{epstopdf}
\usepackage{algorithmic}
\usepackage{float}
\restylefloat{table}
\usepackage{placeins}
\usepackage{array}
\ifpdf
  \DeclareGraphicsExtensions{.eps,.pdf,.png,.jpg}
\else
  \DeclareGraphicsExtensions{.eps}
\fi

\usepackage{subcaption}

\theoremstyle{plain}
    \newtheorem{thm}{Theorem}[section]

    \newtheorem{lemma}[thm]{Lemma}
    
    \newtheorem{conjecture}[thm]{Conjecture}
    
    \newtheorem{theorem}[thm]{Theorem}

\theoremstyle{definition}
    \newtheorem{definition}[thm]{Definition}

    \newtheorem{remark}[thm]{Remark}
\theoremstyle{remark}
    
    \newtheorem{example}[thm]{Example}

\newcommand{\authorfootnotes}{\renewcommand\thefootnote{\@fnsymbol\c@footnote}}

\usepackage{caption} 
\title[Backtracking gradient descent method for general $C^1$ functions]{Backtracking gradient descent method for general $C^1$ functions, with applications to Deep Learning}

\author{Tuyen Trung Truong}
\author{Tuan Hang Nguyen}
\address{Department of Mathematics, The University of Oslo, Blindern, 0316 Oslo, Norway}
\email{tuyentt@math.uio.no}
\address{Axon research, AXON}
\email{hnguyen@axon.com}
\thanks{}
\date{\today}
\begin{document}

\maketitle

\begin{abstract}
 While Standard gradient descent is one very popular optimisation method, its convergence cannot be proven beyond the class of functions whose gradient is globally Lipschitz continuous. As such, it is not actually applicable to realistic applications such as Deep Neural Networks. In this paper, we prove that its backtracking variant behaves very nicely, in particular convergence can be shown for most functions (including all Morse functions). The main theoretical result of this paper is as follows.  

{\bf Theorem.}  Let $f:\mathbb{R}^k\rightarrow \mathbb{R}$ be a $C^1$ function, and $\{z_n\}$ a sequence constructed from the Backtracking gradient descent algorithm. (1) Either $\lim _{n\rightarrow\infty}||z_n||=\infty$ or $\lim _{n\rightarrow\infty}||z_{n+1}-z_n||=0$. (2) Assume that $f$ has at most countably many critical points. Then either $\lim _{n\rightarrow\infty}||z_n||=\infty$ or $\{z_n\}$ converges to a critical point of $f$.  (3) More generally, assume that all connected components of the set of critical points of $f$ are compact. Then either $\lim _{n\rightarrow\infty}||z_n||=\infty$ or $\{z_n\}$ is bounded. Moreover, in the latter case the set of cluster points of $\{z_n\}$ is connected.  

Some generalised versions of this result, including an inexact one where the search direction $v_n$ in $z_{n+1}-z_n=\delta _n v_n$ is only assumed to satisfy Armijo's condition and close to $\nabla f(x_n)$ in a certain sense, are included. The inexact version is then applied to propose backtracking versions for the popular algorithms MMT (momentum) and NAG (Nesterov accelerated gradient), which to our knowledge are new and for which convergence can be proven under assumptions more general than for the standard versions.  Another result in this paper concerns the problem of saddle points.  We then present a heuristic argument to explain why Standard gradient descent method works so well. The heuristic argument leads to modifications of the backtracking versions of GD (gradient descent), MMT and NAG, to  use in Deep Learning.  Experiments with datasets CIFAR10 and CIFAR100 on various popular architectures verify the heuristic argument also for the mini-batch practice and show that our new algorithms, while automatically fine tuning learning rates,  perform better than current state-of-the-art methods such as MMT, NAG, Adagrad, Adadelta, RMSProp, Adam and Adamax. Accompanying source codes are available on GitHub.

\end{abstract}



\maketitle

\section{Introduction}
Let $f:\mathbb{R}^m\rightarrow \mathbb{R}$ be a $C^1$ function, that is continuously differentiable. Denote by $\nabla f $ the gradient of $f$, that is for $x\in \mathbb{R}^m$:
\begin{eqnarray*}
\nabla f(x)=(\frac{\partial f}{\partial x_1}(x),\ldots , \frac{\partial f}{\partial x_m}(x)). 
\end{eqnarray*}
We also denote by $||\nabla f(x)||$ the usual Euclidean norm of the vector $\nabla f(x)\in \mathbb{R}^m$, i.e.
\begin{eqnarray*}
||\nabla f(x)||=\sqrt{[\frac{\partial f}{\partial x_1}(x)]^2+\ldots +[\frac{\partial f}{\partial x_m}(x)]^2}. 
\end{eqnarray*}

Finding critical points (in particular, minima) of such a function is an important problem faced often in science and technology.  Some most common methods are gradient(GD) and Newton's method. While Newton's method works better than GD in the local setting, it requires much stronger assumptions and not many general convergence results in the global setting are proven for it. The purpose of this paper is to demonstrate, both theoretically and experimentally, the advantage of GD (more precisely, a variant of it called Backtracking GD). 

GD was proposed by Cauchy \cite{cauchy} in 1847 to solve systems of non-linear equations, and has been extensively studied for many years since. It has many applications in science and technology, such as being main stream in artificial neural networks in machine learning (for an overview see \cite{bottou-etal, nielsen}) - which has obtained some specular achievements recently such as Alpha Go (in the Game of Go) and driverless cars. Three pioneers of Deep Neural Networks (DNN) were awarded the 2018 Turing prize, considered as the "Nobel prize of computing". 

The general scheme for this approximation method GD is as follows. We start from a {\bf random} point $z_0\in \mathbb{R}^m$, and then construct iteratively
\begin{equation}
z_{n+1}=z_n-\delta _n\nabla f(z_n),
\label{EquationGradientDescentGeneral}\end{equation}
where $\delta _n>0$ are appropriately chosen. In the literature, it is common to call $\delta _n$ learning rates or step sizes. The hope is that $x_n$ will converge to a (global) minimum. The intuition is taken from the familiar picture one obtains when $f$ is a convex function. Note that $z_0$ being random is important here: In \cite{Nesterov2}, one can find a function in $2$ variables and a specific choice of the point $z_0$ for which any sequence as in (\ref{EquationGradientDescentGeneral}), if converges to a critical point at all, can only converge to a saddle point. 

The most known GD method is Standard GD, where we  choose $\delta _n=\delta _0$ for all $n$. Hence, we start with a (randomly chosen) point $z_0\in \mathbb{R}^m$ and define the sequence 
\begin{equation}
z_{n}=z_{n-1}-\delta _0\nabla f(z_{n-1}).
\label{Equation1}\end{equation}  
Rigorous results have been proven for Standard GD in case the gradient $\nabla f$ is globally Lipschitz \cite{curry, crockett-chernoff, goldstein, armijo} and the learning rate $\delta _0$ is small enough ($<1/(2L)$), plus some additional assumptions, see summary in Section \ref{SectionComparison} below. (Moreover, if $f$ is a $C^2$ function then the rate of convergence can be obtained in terms of the local norm of the Hessian $\nabla ^2f$ of $f$.) These assumptions are basically the best under which convergence of Standard GD can be proven, see Examples \ref{LemmaHolderContinuity} and \ref{Example1} below. 

For ease of later reference, we denote by $C^{1,1}_L$ the set of $C^1$ functions $f$ whose gradient is globally Lipschitz with Lipschitz constant $L$, that is $||\nabla f(x)-\nabla f(y)||\leq L||x-y||$ for all $x,y\in \mathbb{R}^m$. It is worthy to remark that the class $C^{1,1}_L$ is not preserved under small perturbations. For example, as mentioned in Example \ref{Example7} below, there is a common technique in DNN called L2 regularisation (or compensation), designed to prevent overfitting, which for a function $f$ considers perturbations of the form $\hat{f}(x)=f(x)+\epsilon ||x||^2$ for some $\epsilon >0$. However, it can be easily checked that if $f\in C^{1,1}_L$, then $\hat{f}$ is not in $C^{1,1}_{L'}$ for every $L'$! More strongly, for all, except  at most one value of $\epsilon \in \mathbb{R}$, the regularisation $\hat{f}$ is not in $C^{1,1}_{L'}$ for every $L'$.

There are many variants of GD (including one by Wolfe \cite{wolfe}) - from most basic to very sophisticated ones, each with a specific range of problems for which it works better than another - have been proposed and studied (for an overview see \cite{ruder}). These will be described in some more detail, in comparison to Backtracking GD, in Section \ref{SectionComparison}  below. At this point, it suffices to mention that as far as we know, all major convergence results in previous work require quite strong assumptions, which are not preserved under small perturbations, such as functions being in $C^{1,1}_L$ plus additional assumptions. In particular, the assumptions under which those results need are not valid for the common L2 regularisation, designed to prevent overfitting. In contrast, our results are valid for most functions, in particular for all Morse functions  (for precise definition, see Section 2) - which are preserved under small perturbations.  

Before going further, we provide a brief description of how Deep Learning is employed in practice. When given a problem (for example, classifying hand written digits), a researcher must choose first of all  (based on their knowledge and experience) an appropriate model to solve it. This is due to the fact that there is no universal model which works uniformly well for all problems ("No-free-lunch-theorem", see \cite{wolpert-mcready}  and a variant in Section 5.1 in \cite{shwartz-david}). If the researcher has chosen a DNN as the model, correspondingly he/she chooses an architecture for his/her model. In a nutshell, this gives rise to a real function $f(x,y,\kappa )$ (which is a composition of simpler real functions, each coming from a layer in the neural network), where $x$ represents one element in the training set, $y$ is  the correct response, and $\kappa $ is a set of parameters which the researcher will later want to optimise upon. For example, in classifying hand written digits, a value for $x$ will be the pixel values stored in the computer for a picture of a hand written number, $y\in J=\{0,1,\ldots ,9\}$ is the correct label for that picture. Under the assumption of conditional independence of the observations, we arrive at the optimisation problem which needs to be solved:  $\hat{\kappa}=argmin _{\kappa}[F(\kappa )=\frac{1}{\sharp I}\sum _{i\in I}f(x_i,y_i,\kappa )]$, where $F$ is the cost (or objective) function and $\sharp A=$ the cardinality of $A$. Closed form solutions for this problem is not feasible for DNN, hence one needs to use a numerical optimisation method, which may introduce yet new parameters usually called  hyper- (or tuning-) parameters. For example, in the standard gradient descent method,  learning rate is a hyper-parameter. In practice, usually one can only work with a small chunk (a mini-batch $I_h$, see the next paragraph for explanation) of the training set at a time. This gives rise to a partial sum $F_h(\kappa  )=\frac{1}{\sharp I_h}\sum _{i\in I_h}f(x_i,y_i,\kappa )$. Hence, one actually work with a sequence of unbiased estimates of the objective $F(\kappa )$, namely $F_h(\kappa )$. The intuition is that the gradient of $F_h$ is typically an unbiased estimator of the gradient of $F$. The trained parameters $\hat{\kappa}$ are used by the DNN to provide predictions when a new datum $x$ is given.  

In practice, we will shuffle randomly the training set and partition it into mini-batches, and when finish working with all mini-batches in one such partition we say that we complete an epoch. To train a DNN, usually one needs to run about tens to hundreds of epochs. Mini-batch training is now main stream in the Deep Learning community, and the reason for using mini-batch comes from the limitation of resources (especially RAM for GPUs) and/or time reduction to compute gradients for each iteration. Computing gradients for full batch takes much more time in comparison to mini-batch, the cost is nearly linear. Moreover, mini-batch can help to escape saddle points, as mentioned in Section 2. In practical deep learning, small mini-batching can cause noise and instability due to covariance between batches (\cite{Jastrzebski}), and the common mini-batch sizes are between $50$ and $256$ \cite{ruder}. We note that our new methods in this paper work stably for many batch sizes, from big batch sizes like $50$ and $256$ as above to small batch sizes such as $10$, see Subsection \ref{Experiment2}.  

The inexact version of GD, which is more suitable to realistic applications such as DNN, is as follows: 
\begin{equation}
z_{n+1}=z_n-\delta _nv_n,
\label{EquationGDInexact}\end{equation}
where we assume that $v_n$ is not too far from the gradient direction $\nabla f(z_n)$. There are many ways to specify the condition of "not too far" (see e.g. \cite{bottou-etal, bertsekas}), here in this paper we use the following common version: there are $A_1,A_2>0$ and $1\geq \mu >0$ such that for all $n$
\begin{eqnarray}
A_1||\nabla f(z_n)||\leq ||v_n||&\leq&A_2||\nabla f(z_n)||,\label{EquationClosenessCondition1} \\
<\nabla f(z_n),v_n>&\geq& \mu ||\nabla f(z_n)||\times ||v_n||.  \label{EquationClosenessCondition2}
\end{eqnarray}
When $\mu=1$ we have that $v_n$ is parallel to $\nabla f(z_n)$ for all $n$, and thus recover the general scheme for GD in (\ref{EquationGradientDescentGeneral}). The geometric meaning of (\ref{EquationClosenessCondition2}) is that the cosine of the angles between $\nabla f(z_n)$ and $v_n$ are positive and bounded from $0$, uniformly in $n$.  

For justification of GD in Deep Learning, a usually cited result is Stochastic GD, which goes back to the work by Robbins and Monro (\cite{robbins-monro, bottou-etal}). However, as will be seen from the analysis in Section \ref{SectionComparison}, at current there is a considerable gap between results which can be proven in the deterministic case (for a single cost function) and in the stochastic case (for example, for the common practice of using mini-batches), which will need to be thoroughly addressed if one wants to have a firm theoretical foundation to justify the use of GD in Deep Learning (in particular, in the rise of adversarial images, sounds and texts \cite{papernot-etal, eykholt-etal}). 

The research in this field is very extensive, and by no means the references in this paper fully represent the state-of-the-art.  For some surveys and implementations of these methods, see e.g. \cite{bottou-etal, ruder, bertsekas, nielsen, boyd-vandenberghe}.  

The focus in this paper is the variant of GD which is called Backtracking GD (and also Backtracking Momentum and Backtracking Nesterov accelerated gradient), because it is related to the backtracking line search method (see Section 2 for precise definition). We will show that this method behaves very nicely, both theoretically and experimentally. In particular, in contrast to Standard GD, convergence can now be proven for Backtracking GD for most functions (including all Morse functions). This expands the class of cost functions for which GD is guaranteed to work, and hence the models, problems and questions which can be dealt with, which is beneficial in particular in view of the No-free-lunch theorem mentioned above. Also, it can be implemented very efficiently in DNN, and experiments on the CIFAR10 and CIFAR100 dataset with various popular architectures shows that the performance of our new methods, while avoiding manually fine-tuning of learning rates, is better than state-of-the-art algorithms such as MMT (Momentum), NAG (Nesterov accelerated gradient),  Adam, Adagrad...  

This paper is organised as follows. In the next section, we will state our main theoretical results and generalisations (including an inexact version of Backtracking GD) with applications to backtracking versions for MMT and NAG, a heuristic argument for the effectiveness of Standard GD and a modification of Backtracking GD (Two-way Backtracking GD) and a brief summary of main experimental results, together with comparisons to previous work in the literature. Proofs of main theoretical results are presented in Section 3, and details of experimental results are presented in Section 4. In the last section, we present conclusions and some open questions. 

{\bf Source codes.} Accompanying source codes for the experiments in Section 4 are available at the following GitHub link \cite{mbtoptimizer}.

{\bf Acknowledgments.} We are grateful to Geir Dahl, who pointed out relevant references, including \cite{bertsekas}. T. T. Truong would like to thank the organisers and participants of the Mathematics in Industry Study Group meeting in 2016 at University of South Australia, in particular the DST group, for exposing him to various real life applications of numerical methods and industrial mathematics (including the gradient descent method). He also appreciates very much the many discussions with Neeraj Kashyap (Google). He thanks also Terje Kvernes and the IT support  team (Department of Mathematics, University of Oslo) for helping with technical issues. 

\section{Statement of the results and comparison to previous work}

In this section, we will first introduce our main theoretical results and generalisations (including an inexact version of Backtracking GD) with applications to backtracking versions of MMT and NAG, and experimental results, including a heuristic argument for the effectiveness of Standard GD and some modifications of Backtracking GD.  Then we compare our results to previous work. 

\subsection{Main theoretical results} We recall first the so-called Armijo's condition, for some $0<\alpha < 1$ and some $x,y\in \mathbb{R}^m$ (the main case of interest is when $y=x-\sigma \nabla f(x)$):
\begin{equation}
f(y)-f(x)\leq \alpha <\nabla f(x),y-x>, 
\label{EquationArmijoCondition}\end{equation}
here $<.,.>$ is the standard inner product in $\mathbb{R}^m$. Given as above $f:\mathbb{R}^m\rightarrow \mathbb{R}$ a $C^1$-function, $\delta _0>0$ and $1>\alpha ,\beta >0$ arbitrary numbers,  we define the function $\delta (f,\delta _0,x):\mathbb{R}^m\rightarrow \mathbb{R}$ to be the largest number $\sigma $ among $\{\beta ^n\delta _0:~n=0,1,2,\ldots \}$ for which 
\begin{equation}
f(x-\sigma \nabla f(x))-f(x)\leq -{\alpha \sigma }||\nabla f(x)||^2.
\label{Equation2}\end{equation} 
By Lemma \ref{Lemma2} below, the function $\delta (f,\delta _0,x)$ is well-defined and is always positive. Moreover, this positivity is uniform on compact subsets of $\mathbb{R}^m$ if $\nabla f$ is locally Lipschitz (\cite{curry, crockett-chernoff, goldstein, armijo}). For a given $z_0\in \mathbb{R}^m$, we define the sequence
\begin{equation}
z_n=z_{n-1}-\delta (f,\delta _0,z_{n-1})\nabla f(z_{n-1}),
\label{Equation3}\end{equation}
for $n=1,2,3\ldots $. 
This update rule is usually called Backtracking GD or Armijo's rule in the literature. 

We recall that a point $z^*$ is a cluster point of a sequence $\{z_n\}$ if there is a subsequence $\{z_{n_k}\}$ so that $\lim _{k\rightarrow\infty}z_{n_k}=z^*$. The sequence $\{z_n\}$ converges if and only if it has one and only one cluster point. It has been known that any cluster point of the sequence $\{z_n\}$ in the Backtracking GD is a critical point of $f$, see e.g. Proposition 1.2.1 in \cite{bertsekas}. (We remark that  \cite{bertsekas} uses the terminology "limit points" instead of  the more common one "cluster points", which may cause some confusion occasionaly.) However, in previous work, starting from \cite{armijo}, convergence for Backtracking GD is proven only under strong assumptions similar to those in Theorem \ref{TheoremArmijoGeneralisation} below, in particular it is required that the function belongs to $C^{1,1}_L$. Our main theoretical result of this paper, stated next, extends this to a much more broad and realistic class of cost functions. 

\begin{theorem} Let $f$ be a $C^1$ function and let $\{z_n\}_{n=0,1,2,\ldots }$ be defined as in (\ref{Equation3}). 

1) Either $\lim _{n\rightarrow\infty}||z_n||=\infty$ or $\lim _{n\rightarrow\infty}||z_{n+1}-z_n||=0$. 

2) Assume that $f$ has at most countably many critical points. Then either\\ $\lim _{n\rightarrow\infty}||z_n||$ $=\infty$ or $\{z_n\}$ converges to a critical point of $f$.  

3) More generally, assume that all connected components of the set of critical points of $f$ are compact. Then either $\lim _{n\rightarrow\infty}||z_n||=\infty$ or $\{z_n\}$ is bounded. Moreover, in the latter case the set of cluster points of $\{z_n\}$ is connected. 
\label{Theorem1}\end{theorem}
A generalisation of Theorem \ref{Theorem1}, where we allow the vector $v_n$ in $\delta _nv_n=z_{n+1}-z_n$ to not be parallel with $\nabla f(z_n)$ but only satisfies conditions (\ref{EquationClosenessCondition1}) and (\ref{EquationClosenessCondition2}), which is more relevant to the mini-batch practice, will be provided in Theorem \ref{TheoremSequenceVersion}. The proofs of the two theorems are the same. However, since the statement of Theorem \ref{TheoremSequenceVersion}  is much more complicated than that of Theorem \ref{Theorem1}, we choose to state Theorem \ref{Theorem1} first for easy understanding. 
\begin{remark}

Note that the proof of Theorem \ref{Theorem1} needs only that $\delta (f,\delta _0,z)$ $\leq$  $\delta _0$ for all $z$, that any cluster point of $\{z_n\}$  is a critical point of $f$, and that Armijo's condition is satisfied (and not the specific definition of $\delta (f,\delta _0,z)$). Hence, the conclusion of the theorem holds in more general settings, for example under which Wolfe's conditions provide sequences $\{z_n\}$ for which $\{\nabla f(z_n)\}$ converges to $0$ - for detail see Subsection \ref{SectionComparison}. This will be used later to provide modifications of Theorem \ref{Theorem1}  which are more suitable to realistic applications, including backtracking versions of MMT and NAG. 

In the theorem, note that a general sequence $\{z_n\}$ may have both a convergent subsequence and another subsequence diverging to infinity. Hence, it can be regarded as a miracle that Backtracking GD guarantees the conclusions of the theorem. For example, for Standard GD, all these conclusions fail in general, see Subsection \ref{SectionComparison} for detail.  

If $f$ is a function without critical points, for example $f(t)=e^t$, then the sequence $\{z_n\}$ cannot have any cluster point inside $\mathbb{R}^n$. Therefore, it must diverge to infinity. The same argument applies  more generally to  $f$ and $z_0$ for which the set $\{z\in \mathbb{R}^m:~f(z)\leq f(z_0)\}$ contains no critical points of $f$. See \cite{bertsekas} for more detail. 

The assumption in part 2 of Theorem \ref{Theorem1} is satisfied by all Morse functions. We recall that a $C^2$ function $f$ is Morse if all of its critical points are non-degenerate. This condition means that whenever $x^*\in \mathbb{R}^m$ is a  critical point of $f$, then the Hessian matrix $\nabla ^2f(x^*)=(\partial ^2f/\partial x_i\partial x_j)_{i,j=1,\ldots ,m}(x^*)$ is invertible. All critical points of a Morse function are isolated, and hence there are at most countably many of them. Moreover, note that Morse functions are dense. In fact, given any $C^2$ function $g$, the function $f(x)=g(x)+<a,x>$ is Morse for $a$ outside a set of Lebesugue's measure $0$, by Sard's lemma. Hence Morse functions are dense in the class of all functions. For example, $g(x)=x^3$ is not a Morse function, but $f(x)=g(x)+ax$ is Morse for all $a\not= 0$. More stronger, by using transversality results, it can be shown that the set of all Morse functions is preserved under small perturbations.  

In case $f$ is real analytic, then without the assumption that $f$ has at most countably many critical points,  \cite{absil-mahony-andrews} also showed that the sequence $\{z_n\}$ either diverges to infinity or converges. However, the real analytic assumption is quite restrictive.   

In practice, knowing that the sequence in (\ref{Equation3}) converges, and that the limit point is moreover a local minimum point, is usually good enough. In fact, for linear networks, it is known that every local minimum is a global minimum \cite{kawaguchi}. In contrast to conventional wisdom derived from low dimensional intuition, a working hypothesis in Deep Learning is that local minima with high error (that is, a minimum point at which the value of the function is too much bigger than the optimal value) are exponentially rare in higher dimensions \cite{dauphin-pascanu-gulcehre-cho-ganguli-bengjo, swirszcz-czarnecki-pascanu}, and hence finding local minima is enough. Perturbation can help to escape from saddle points  \cite{ge-huang-jin-yuan}, and hence we expect that backtracking gradient descent for random mini-batches will also help to escape saddle points in training neural networks.  See also Theorem \ref{Theorem2} below.   
\label{ExampleBestPossible}\end{remark}

In practical applications, we would like the sequence $\{z_n\}$ to converge to a minimum point.  It has been shown in \cite{dauphin-pascanu-gulcehre-cho-ganguli-bengjo} via experiments that for cost functions appearing in DNN the ratio between minima and other types of critical points becomes exponentially small when the dimension $m$ increases, which illustrates a theoretical result for generic functions  \cite{bray-dean}. Which leads to the question: Would in most cases GD  converge to a minimum? We will see that because it is indeed a {\bf descent} method, Backtracking GD answers the above question in affirmative in a certain sense. Before stating the precise result, we formalise the notion of "a critical point which is not a minimum". 

{\bf Generalised saddle point.} Let $f$ be a $C^1$ function and let $z_{\infty}$ be a critical point of $f$. Assume that $f$ is $C^2$ near $z_{\infty}$. We say  that $z_{\infty}$ is a generalised saddle point of $f$ if the Hessian $\nabla ^2f(z_{\infty})$   has at least one {\bf negative} eigenvalue. 

If $z_{\infty}$ is a non-degenerate critical point of $f$, then  it is  a minimum if and only if all eigenvalues of the Hessian $\nabla ^2f(z_{\infty})$ are positive. Hence in this case we see that a critical point of $f$ is a minimum if and only if it is not a generalised saddle point.  If $U\subset \mathbb{R}^m$ is an open set, we denote by $Vol(U)$ its Lebesgue measure. Also, if $\epsilon >0$ and $z_{\infty}\in \mathbb{R}^m$, we denote by $B(z_{\infty}, \epsilon )$ $=$ $\{x\in \mathbb{R}^m: ~||x-z_{\infty}||<\epsilon \}$ the open ball of radius $\epsilon $ and centre $z_{\infty}$. For any critical point $z_{\infty}$ of $f$, we define:

$\mathcal{D}(z_{\infty})$ $=$ $\{z_0\in \mathbb{R}^m:$ the sequence $\{z_n\}$ in (\ref{Equation3}) {\bf does not contain} any subsequence converging to $z_{\infty}\}$. 

\begin{theorem} Let $f$ be a $C^1$ function. Assume that $z_{\infty}$ is a saddle point of $f$. For every $\epsilon >0$, there is a non-empty open set $U(z_{\infty},\epsilon )\subset \mathcal{D}(z_{\infty})\cap B(z_{\infty}, \epsilon )$. Moreover, we have the following density 1 property:
\begin{eqnarray*}
\lim _{\epsilon\rightarrow 0}\frac{Vol(U(z_{\infty},\epsilon ))}{Vol(B(z_{\infty}, \epsilon ))}=1. 
\end{eqnarray*}

\label{Theorem2}\end{theorem} 

\begin{remark}

1) It is known that if the starting point $z_0$ of Backtracking GD is close enough to an isolated local minimum point, then the sequence $\{z_n\}$ will converge to that local minimum point. See e.g.  Proposition 1.2.5 (Capture Theorem) in \cite{bertsekas}.   

2) Theorem \ref{Theorem2} asserts roughly that if $\lim _{n\rightarrow\infty}z_n=z_{\infty}$, then it is very rare that $z_{\infty}$ is a saddle point.  In the proof of Theorem \ref{Theorem2}, the open sets $U(z_{\infty}, \epsilon )$ are chosen to be cones. 

3) It is observed that in practical problems in artificial neural networks, the cost function $f$ has approximately the same value at all local minimum points. Hence, finding any local minimum is usually good enough for practical purposes. Moreover (\cite{dauphin-pascanu-gulcehre-cho-ganguli-bengjo}), for higher-dimensional functions we expect that saddle points are prevalent, and it is very difficult to escape saddle points. Hence, having a density 1 property as in Theorem \ref{Theorem2} is good toward this concern. 

 \label{Remark2}\end{remark}

\subsection{Some generalisations of Theorem \ref{Theorem1}}

The following is a modification of Theorem \ref{Theorem1} which is more suitable to realistic applications. 
\begin{theorem} Let $f:\mathbb{R}^m\rightarrow \mathbb{R}$ be a $C^1$ function which has at most countably many critical points. Let $z_0$ be a point in $\mathbb{R}^m$, $0<\alpha ,\beta  <1$ and $\delta _0>0$.

1) Assume that $\nabla f$ is locally Lipschitz near every critical points of $f$. Let $\{z_n\}_{n\geq 0 }$ be the sequence constructed from Backtracking GD update rule (\ref{Equation3}). Then either $\lim _{n\rightarrow\infty}||z_n||=\infty$ or $\inf _{n\geq 0 }\delta (f,\delta _0,z_n)>0$. 

2) Conversely, let $\{\delta _n\}_{n\geq 0 }$ be a sequence of real numbers in $(0,\delta _0)$ so that $\inf _{n}\delta _n$ $>0$. Define the sequence $\{z_n\}$ by the update rule: $z_n=z_{n-1}-\delta _{n-1}\nabla f(z_{n-1})$ for $n\geq 1$. Assume moreover that for every $n\geq 1$ we have $f(z_n)-f(z_{n-1})\leq -\alpha \delta _{n-1}||\nabla f(z_{n-1})||^2$. Then either $\lim _{n\rightarrow\infty}||z_n||=\infty$ or $\{z_n\}$ converges to a critical point of $f$.   

\label{TheoremBoundednessLearningRates}\end{theorem}
\begin{proof}
1) Assume that there is a subsequence of $\{z_n\}$ which is bounded. Then from Theorem \ref{Theorem1} we have that the sequence $\{z_n\}$ converges to a critical point $z_{\infty}$ of $f$. By assumption on $f$, there is an open neighbourhood $B(z_{\infty},r)$ of $z_0$ on which $\nabla f$ is locally Lipschitz. Then it is well-known that $\inf _{z\in B(z_{\infty},r)}\delta (f,\delta _0,z)>0$. Choose $N$ be an integer such that $z_n\in B(z_{\infty},r)$ for all $n\geq N$. We then have
\begin{eqnarray*}
\inf _{n=0,1,\ldots }\delta (f,\delta _0,z_n)\geq \min \{\inf _{0\leq n\leq N}\delta (f,\delta _0,z_n),\inf _{z\in B(z_{\infty},r)}\delta (f,\delta _0,z) \}>0.
\end{eqnarray*}

2) It can be checked that under the assumption $\inf _{n=0,1,\ldots }\delta _n >0$, the proof of Proposition 1.2.1 in \cite{bertsekas} goes through and shows that any cluster point of the sequence $\{z_n\}$ must be a critical point of $f$. Then from Remark \ref{ExampleBestPossible} it follows that proofs of all parts of Theorem \ref{Theorem1} go through and give us the desired result.   
\end{proof}
\begin{remark}
Example \ref{LemmaHolderContinuity} shows that the assumption on $\nabla f$ being locally Lipschitz near critical points of $f$ is necessary for part 1 of Theorem \ref{TheoremBoundednessLearningRates} to hold. 

Also, the assumption $\inf _{n}\delta _n>0$ is needed in general for part 2 of Theorem \ref{TheoremBoundednessLearningRates} to hold. In fact, let $f(x)=x^2$ and choose any $z_0>0$. If we choose any sequence of positive numbers $\{\delta _n\}$ so that $\sum _{n=0}^{\infty}\delta _n$ is small enough (depending on $z_0,\alpha $), then the sequence $\{z_n\}$ defined by the update rule $z_n=z_{n-1}-\delta _{n-1}f'(z_{n-1})$ for $n\geq 1$ will satisfy $f(z_n)-f(z_{n-1})\leq -\alpha \delta _{n-1}|f'(z_{n-1})|^2$ for all $n\geq 1$ and $\lim _{n\rightarrow \infty}z_n=z_{\infty}$  exists, but this point $z_{\infty}$ is $>0$, and hence is not the unique critical point $0$ of $f$. 
\label{RemarkTheoremBoundednessLearningRates}\end{remark}

The following generalisation of Theorem \ref{Theorem1} is relevant to the practice of using mini-batches in DNN. It uses the inexact version of Backtracking GD which we now introduce. 

{\bf Inexact Backtracking GD.} Fix $A_1,A_2 , \delta _0>0$; $1\geq \mu >0$ and $1>\alpha , \beta>0$. We start with a point $z_0$ and define the sequence $\{z_n\}$ by the following procedure. At step $n$:

i) Choose a vector $v_n$ satisfying $A_1||\nabla f(z_n)||\leq ||v_n||\leq A_2||\nabla f(z_n)||$ and $<\nabla f(z_n),v_n>\geq \mu ||\nabla f(z_n)||\times ||v_n||$. (These are the same as conditions (\ref{EquationClosenessCondition1}) and (\ref{EquationClosenessCondition2}) for Inexact GD.)

ii) Choose $\delta _n$ to be the largest number $\sigma $ among $\{\delta _0, \delta _0\beta ,\delta _0\beta ^2,\ldots \}$ so that 
\begin{eqnarray*}
f(z_n-\sigma v_n)-f(z_n)\leq -\alpha \sigma <f(z_n),v_n>.
\end{eqnarray*}
(This is Armijo's condition (\ref{EquationArmijoCondition}).) 

iii) Define $z_{n+1}=z_n-\delta _nv_n$. 

As in Lemma \ref{Lemma2}, we can choose $\delta _n$ to be bounded on any compact set $K$ on which $\nabla f$ is nowhere zero. 

\begin{theorem}
Let $f:\mathbb{R}^m\rightarrow \mathbb{R}$ be a $C^1$ function and let $z_n$ be a sequence constructed from the Inexact Backtracking GD procedure.  

1) Either $\lim _{n\rightarrow\infty}||z_n||=\infty$ or $\lim _{n\rightarrow\infty}||z_{n+1}-z_n||=0$. 

2) Assume that $f$ has at most countably many critical points. Then either\\ $\lim _{n\rightarrow\infty}||z_n||$ $=\infty$ or $\{z_n\}$ converges to a critical point of $f$.  

3) More generally, assume that all connected components of the set of critical points of $f$ are compact. Then either $\lim _{n\rightarrow\infty}||z_n||=\infty$ or $\{z_n\}$ is bounded. Moreover, in the latter case the set of cluster points of $\{z_n\}$ is connected. 

\label{TheoremSequenceVersion}\end{theorem} 

The proof of this theorem is identical to that of Theorem \ref{Theorem1}. Below are some other practical modifications. 

\begin{example} (Sequence of cost functions) Another generalisation of Theorem \ref{Theorem1} is as follows. Let $f_n$ converge uniformly on compact sets to $f$, so that $\nabla f_n$ also converges uniformly on compact sets to $\nabla f$. Consider the following sequence: $z_{n+1}=z_n-\delta _n\nabla f_n(z_n)$, where $\delta _n=\delta (f_n,\delta _0,z_n)$. Then all conclusions of Theorem \ref{Theorem1} are satisfied.   
\label{Example6}\end{example}

\begin{example} (Non-smooth functions) Analysing the proofs of the conclusions in Theorem \ref{Theorem1}, we can see that they do not need the assumption that $f(x)$ is $C^1$. For example, assuming only that $f$ is directionally differentiable and $\nabla f(x)$ is locally bounded is enough. Even we may consider more general notions of derivative and boundedness, such as Lebesgue's integral and derivative, and essential maximum and minimum. The function $f(x)=|x|$ is not differentiable at $0$, but we can check that the Fundamental Theorem of Calculus still applies by using Lebesgue's integral, and hence can check that Theorem \ref{Theorem1} is valid also for this function. 

Another direction to generalise Theorem \ref{Theorem1} is to consider functions on open subsets $\Omega$ of $\mathbb{R}^m$. In this case, we have basically the same conclusion, if we either stop the iteration or choose a disturbance $v_n$ of $\nabla f(z_n)$ so that $z_n-\delta _nv_n\in \Omega $ whenever the point $z_n$ is on the boundary of $\Omega$. Theorem \ref{TheoremSequenceVersion} can be used to justify for this algorithm.  
\label{Example5}\end{example}

\begin{example} (Regularisation) There is a common practice of using regularisation in machine learning, that is we consider instead of a cost function $f$, the compensated version $g(x)=f(x)+\lambda ||x||^2$ (for the L2 regularisation), where $\lambda >0$ is a constant. It is observed that usually working with the regularisation cost function $g(x)$ gives better convergence and prevent overfitting (see e.g. \cite{nielsen}) . Theoretically, regularisation is also good in light of Theorem \ref{Theorem1} and the fact mentioned above that Morse functions are dense.  Also, we may consider L1 regularisation $g(x)=f(x)+\lambda \sum _i|x_i|$, provided we are careful about using Lebesgue's integral and derivatives as in Example \ref{Example5}. 
\label{Example7}\end{example}

\subsection{Backtracking versions of MMT and NAG}\label{SectionBacktrackingMMTNAG}  While saddle points are in general not a problem for both Standard GD and Backtracking GD, it is not rare for these algorithms to converge to bad local minima. MMT and NAG are popular methods aiming to avoid bad local minima. For details, the readers can consult \cite{ruder}.

For the standard version of MMT, we fix two numbers $\gamma, \delta >0$, choose two initial points $z_0,v_{-1}\in \mathbb{R}^m$ and use the following update rule: 
\begin{eqnarray*}
v_n&=&\gamma v_{n-1}+\delta \nabla f(z_n),\\
z_{n+1}&=&z_n-v_n.
\end{eqnarray*}

The standard version of NAG is a small modification of MMT: we fix again two numbers $\gamma, \delta >0$, choose two initial points $z_0,v_{-1}\in \mathbb{R}^m$, and use the update rule: 
\begin{eqnarray*}
v_n&=&\gamma v_{n-1}+\delta \nabla f(z_n-\gamma v_{n-1}),\\
z_{n+1}&=&z_n-v_n.
\end{eqnarray*}
 
If $\gamma =0$, both MMT and NAG reduce to Standard GD.  While observed to work quite well in practice, the convergence of these methods are not proven for functions which are not in $C^{1,1}_L$ or not convex. For example, the proof for convergence of NAG in Section 2.2 in \cite{Nesterov2} requires that the function $f$ is in $C^{1,1}_L$ and is moreover strongly convex. (For this class of functions, it is proven that NAG achieves the best possible convergence rate among all gradient descent methods.) Therefore, it is seen that convergence results for these methods require even stronger assumptions than that of Standard GD, see Subsection \ref{SectionComparison}  below. 

Here, inspired by the Inexact Backtracking GD, we propose the following backtracking versions of MMT and NAG, whose convergence can be proven for more general functions. As far as we know, these backtracking versions are new. 

{\bf Backtracking MMT.} We fix $0<A_1<1<A_2$; $\delta _0, \gamma _0>0$; $1\geq \mu >0$ and $1>\alpha ,\beta >0$, and choose initial points $z_0,v_{-1}\in \mathbb{R}^m$. We construct sequences $\{v_n\}$ and $\{z_n\}$ by the following update rule: 
\begin{eqnarray*}
v_n&=&\gamma _nv_{n-1}+\delta _n\nabla f(z_n),\\
z_{n+1}&=&z_n-v_n.
\end{eqnarray*}
Here the values $\gamma _n$ and $\delta _n$ are chosen in a two-step induction as follows. We start with $\sigma =1$, $\gamma '=\gamma _0$ and $\delta '=\delta _0$ and $v_n=\gamma 'v_{n-1}+\delta '\nabla f(z_n-\gamma v_{n-1})$ and $z_{n+1}=z_n-\sigma v_n$.   While condition i) for Inexact Backtracking GD is not satisfied, we replace  $\gamma '$ by $\gamma '\beta$ and update $v_n$ and $z_{n+1}$ correspondingly.  It can be easily checked that after a finite number of such steps, condition i) in Backtracking GD will be satisfied.  At this point we obtain values $\gamma _n'$ and $\delta _n'$. Then, while condition ii) in Inexact Backtracking GD is not satisfied, we replace $\sigma$ by $\sigma \beta$ and update $v_n$ and $z_{n+1}$ correspondingly.  It can be easily checked that after a finite number of such steps, condition ii) in Backtracking GD will be satisfied. At this point we obtain a value $\sigma _n$. Then we choose $\gamma _n=\sigma _n\gamma _n'$ and $\delta _n=\sigma _n\delta _n'$. 

{\bf Backtracking NAG.} The update rule is similar to that for Backtracking MMT. 

If $\gamma _0=0$,  both Backtracking MMT and Backtracking NAG reduce to Backtracking GD. We have the following result for the convergence of these backtracking versions. 

\begin{theorem} Let $f:\mathbb{R}^m\rightarrow \mathbb{R}$ be in $C^1$. Let $\{z_n\}$ be a sequence constructed by either the Backtracking MMT update rule or the Backtracking NAG update rule. 

1) Either $\lim _{n\rightarrow\infty}||z_n||=\infty$ or $\lim _{n\rightarrow\infty}||z_{n+1}-z_n||=0$. 

2) Assume that $f$ has at most countably many critical points. Then either\\ $\lim _{n\rightarrow\infty}||z_n||$ $=\infty$ or $\{z_n\}$ converges to a critical point of $f$.  

3) More generally, assume that all connected components of the set of critical points of $f$ are compact. Then either $\lim _{n\rightarrow\infty}||z_n||=\infty$ or $\{z_n\}$ is bounded. Moreover, in the latter case the set of cluster points of $\{z_n\}$ is connected. 
\label{TheoremBacktrackingMMTNAG}\end{theorem}
\begin{proof}
In fact, our Backtracking MMT and Backtracking NAG algorithms satisfy conditions for Inexact Backtracking GD, hence Theorem \ref{TheoremSequenceVersion} can be applied to yield the desired conclusions. 
\end{proof}

\subsection{A heuristic argument for the effectiveness of Standard GD}\label{SubsectionHeuristics} Here we give a heuristic argument to show that in practice, for the deterministic case of a single cost function, Backtracking GD will behave like the Standard GD in the long run. Let us consider a function $f:\mathbb{R}^m\rightarrow \mathbb{R}$ which is of class $C^1$, and consider a number $\delta _0>0$ and a point $z_0\in \mathbb{R}^m$. For simplicity, we will assume also that $f$ has only at most countably many critical points (which is the generic case), and that $\nabla f$ is locally Lipschitz at every critical point of $f$ (a reasonable assumption). We then have from Theorem \ref{TheoremBoundednessLearningRates} above that if the sequence $\{z_n\}$ in (\ref{Equation3}) does not contain a sequence diverging to $\infty$, then it will converge to a critical point $z_{\infty}$ of $f$. If we assume moreover that $f$ is $C^2$ near $z_{\infty}$, then we have moreover that $\delta (f,\delta _0,z)>0$ uniformly near $z_{\infty}$. Therefore, since $\delta (f,\delta _0,z)$ takes values in the discrete set $\{\beta ^n\delta _0:~n=0,1,2,\ldots \}$, we have that $\delta (f,\delta _0,z_n)$ is contained in a finite set. Hence the number $\delta _{\infty}:= \inf _{n=1,2,\ldots }\delta _n>0$. Therefore, the Standard GD with learning rate $\delta _0=\delta _{\infty}$ should converge. 

Our experiments with the data set CIFAR10 and CIFAR100 and various different architectures, more details in the next subsection and Section 4, show that this argument is also verified for the practice of using mini-batches in DNN.  In the case where $\nabla F_h(\kappa _h)$ approximates well $\nabla F(\kappa _h)$ and $\delta (F_h,\delta _0,\kappa _h)$ approximates well $\delta (F,\delta _0,\alpha _h)$, then Theorem \ref{TheoremSequenceVersion} can be applied to justify the use of (Inexact) Backtracking GD in DNN. 

\subsection{Two-way Backtracking GD and Main experimental results} In this subsection we present first some modifications of Backtracking GD which are aimed to save time and computations, and then present briefly main experimental results. More detail on experimental results will be presented in Section 4. 

We know that if $z_n$ converges to $z_{\infty}$, and $\delta _0\geq \delta _n>0$ are positive real numbers so that $f(z_n-\delta _n\nabla f(z_n))-f(z_n)\leq -\alpha \delta _n ||\nabla f(z_n)||^2$ for all $n$, then $f(z_{\infty}-\delta _{\infty}\nabla f(z_{\infty}))-f(z_{\infty})\leq -\alpha \delta _{\infty} ||\nabla f(z_{\infty})||$ for $\delta _{\infty}=\limsup _{n\rightarrow \infty}\delta _n$. Moreover, at least for Morse functions, the arguments in the previous subsection show that all the values $\{\delta _n\}$ belong to a finite set. Therefore, intuitively we will save more time by starting for the search of the learning rate  $\delta _n$ not from $\delta _0$ as in the pure Backtracking GD procedure, but from the learning rate $\sigma =\delta _{n-1}$ of the previous step, and allowing increasing $\sigma$, and not just decreasing it, in case $\sigma $  satisfies inequality (\ref{Equation2}) and still does not exceed $\delta _0$. We call this Two-way Backtracking GD. More precisely, it works as follows. At step $n$, choose $\sigma =\delta _{n-1}$. If $\sigma $ does not satisfy (\ref{Equation2}), then replace $\sigma$ by $\beta \sigma $ until (\ref{Equation2}) is satisfied. If $\sigma$ satisfies (\ref{Equation2}), then while $\sigma /\beta $ satisfies (\ref{Equation2}) and $\leq \delta _0$, replace $\sigma $ by $\sigma /\beta$. Then define $\delta _n$ to be the final value $\sigma$. Theorem \ref{TheoremBoundednessLearningRates}  can be used to justify that this procedure should also converge. 

The following example illustrates the advantage of Two-way Backtracking GD, in the deterministic case, compared to the Standard GD and the pure Backtracking GD. 

\begin{example}[Mexican hats]  The Mexican hat example in \cite{absil-mahony-andrews} is as follows (see Equation 2.8 therein). In polar coordinates $z=(r,\theta )$, the function has the form $f(r,\theta )=0$ if $r\geq 1$, and when $r<1$ it has the form
\begin{eqnarray*}
f(r,\theta )=[1-\frac{4r^4}{4r^4+(1-r^2)^4}\sin (\theta -\frac{1}{1-r^2})]e^{-1/(1-r^2)}.
\end{eqnarray*}
For this function,  \cite{absil-mahony-andrews} showed that if we start with an initial point $z_0=(r_0,\theta _0)$, where $\theta _0(1-r_0^2)=1$, then the gradient descent flow $(r(t), \theta (t))$ (solutions to $x'(t)=-\nabla f(x(t))$) neither diverges to infinity nor converges as $t\rightarrow 0$. 

For this example, we have run many experiments with Standard GD, pure Backtracking GD and Two-way Backtracking GD for random choices of initial values $z_0$'s. We found that in contrast to the case of the continuous method, all these three discrete methods {\bf do} converge. We propose the following explanation for this seemingly contradiction. For the continuous method, the gradient descent flow $(r(t),\theta (t))$, with an initial point $z_0$ on the curve $\theta (1-r^2)=1$ will get stuck on this curve, while for the discrete methods right after the first iterate we already escape this curve. 

For this Mexican hat example, we observe that Two-way Backtracking GD works much better than Standard GD and pure Backtracking GD. In fact, for the case where $z_0=(r_0,\theta _0)$ satisfies $\sin (\theta _0)<0$, for a random choice of initial learning rate Standard GD needs many more iterates before we are  close to the limit point than  Two-way Backtracking GD, and pure Backtracking GD needs more total time to run and less stable than Two-way Backtracking GD.     \label{Example8}\end{example}

In our experiments we test the above results and heuristic arguments and also the effectiveness of using them in DNN. We work with datasets CIFAR10 and CIFAR100, on various architectures. The details of these experiments will be described in Section 4. Here we summarise the results. The most special feature of our method is that it achieves high accuracy while being completely automatic: we do not need to use manual fine-tuning as in the current common practice. The automaticity of our method is also different in nature from those in more familiar methods such as Adagrad and Adadelta. More detailed comparison is given in the next subsection.
 
 In practice, we see that the performance of Backtracking GD and Two-way Backtracking GD are almost identical, while the time spent for Backtracking GD is about double (or more) that of time to run Two-way Backtracking GD. Thus, this confirms our intuition that Two-way Backtracking GD helps to save time. Also, it is demonstrated that even here when we do not have just a single cost function but with mini-batches, the average of learning rates (scaled depending on the mini-batch size, see Section 4.1) associated to  mini-batches in the same epoch, when computed with respect to the Backtracking GD scheme, does behave as what argued in our heuristic argument in Subsection \ref{SubsectionHeuristics}. The validation accuracy obtained is also very high and better than state-of-the-art algorithms, see Section 4 for details.

\subsection{Comparison to previous work}\label{SectionComparison} 

In the influential paper \cite{armijo}, where condition (\ref{EquationArmijoCondition}) was introduced into GD, Armijo proved the existence of Standard GD and Backtracking GD for functions in $C^{1,1}_L$, under some further assumptions.  The most general result which can be proven with his method is as follows (see e.g. Proposition 12.6.1 in \cite{lange}).  
\begin{theorem}
Assume that $f$ is in $C^{1,1}_L$, $f$ has compact sublevels (that is, all the sets $\{x:~f(x)\leq b\}$ for $b\in \mathbb{R}$ are compact), and the set of critical points of $f$ is bounded and has at most countably many elements. Let $\delta \leq 1/(2L)$. Then the Standard GD converges to one critical point of $f$. 
\label{TheoremArmijoGeneralisation}\end{theorem}
 An analog of Theorem \ref{TheoremArmijoGeneralisation} for gradient descent flow (solutions to $x'(t)=-\nabla f(x(t))$) is also known (see e.g. Appendix C.12 in \cite{helmke-moore}). Both the assumptions that $f$ is in $C^{1,1}_L$ and $\delta $ is small enough are necessary for the conclusion of Theorem \ref{TheoremArmijoGeneralisation}, as shown by the next two very simple examples. 

\begin{example} (When $\gamma =1/2$, a similar example is given as Exercise 1.2.3 in \cite{bertsekas}, without proof. For the sake of completeness, here we provide a proof.) Let $0<\gamma <1$ be a rational number. Let $f:\mathbb{R}\rightarrow \mathbb{R}$ be the function $f(x)=|x|^{1+\gamma }$. Then it can be checked that $f$ is in $C^1$ and $f'(x)$ is H\"older continuous with the H\"older exponent $\gamma$. (We recall that a Lipschitz function is a H\"older function with H\"older exponent $1$.) Moreover, $0$ is the only critical point of $f$ (and is in fact a global minimum) and  $f$ has compact sublevels. 

Now we show that there is a countable set $A\subset \mathbb{R}$ such that for any $\delta _0>0$ and any $0\not= z_0\in \mathbb{R}\backslash A$, the sequence $\{z_n\}$ in (\ref{Equation1}) does not converge to $0$. To this end, it suffices to show the following: There is a countable set $A\subset \mathbb{R}$ so that if $\delta _n$ is any sequence of positive numbers for which the sequence $\{z_n\}$ in (\ref{EquationGradientDescentGeneral}) with $z_0\in \mathbb{R}\backslash A$ converges to $0$, then $\inf _{n\geq 1}\delta _n=0$.

We write $\gamma =p/q$ for relatively prime positive integers $p$ and $q$. Note that for any $\delta ,y\in \mathbb{R}$ there are at most $2q$ solutions $x$ to $x-\delta f'(x)=y$. Therefore, since $\delta (f,\delta _0,x)$ belongs to the countable set $\{\beta ^n\delta _0:n=0,1,2,\ldots \}$, it follows that there is a countable set $A\subset \mathbb{R}$ so that: if $z_n=0$ for some $n$ then $z_0\in A$. Now choose any $z_0\in \mathbb{R}\backslash A$ so that $z_n$ converges to $0$, we claim that $\delta _1:=\inf _{n}\delta (f,\delta _0,z_n)=0$. Assume otherwise that $\delta _1>0$, we will obtain a contradiction as follows. Note that for all $x\not= 0$, we have $x.f'(x)>0$ and $|f'(x)|=(1+\gamma ) |x|^{\gamma }$. Therefore, if $x\not= 0$ is close to $0$ then $|f'(x)|>>|x|$. Thus, since $\lim _{n\rightarrow\infty}z_n=0$ and $z_n\not= 0 $ for all $n$ (by the assumption on $z_0$), by discarding a finite number of points in the sequence $\{z_n\}$ if necessary, we can assume that 
\begin{eqnarray*}
 |z_{n+1}|=|z_n-\delta (f,\delta _0,z_n)f'(z_n)|\geq \delta _1|z_n|^{\gamma}/2
\end{eqnarray*}
for all $n=0,1,2,\ldots $. Putting $\delta _2=\delta _1/2$, by iterating the above inequality we obtain
\begin{eqnarray*}
 |z_{n}|\geq \delta _2^{1+\gamma +\gamma ^2+\ldots +\gamma ^n}.|z_0|^{\gamma ^n},
\end{eqnarray*}
for all $n$. Taking limit when $n\rightarrow\infty$, we obtain 
\begin{eqnarray*}
0=\lim _{n\rightarrow\infty} |z_n|\geq \delta _2^{1/(1-\gamma )}>0, 
\end{eqnarray*}
which is a contradiction. Therefore, $\inf _{n}\delta (f,\delta _0,z_n)=0$, as claimed. 
\label{LemmaHolderContinuity}\end{example}

In contrast, for the function in the above example, Backtracking GD converges to $0$ for every choice of initial point $z_0$.

\begin{example}  Let $\epsilon _0$ be any positive number. Consider a smooth function $f:\mathbb{R}\rightarrow \mathbb{R}$ which on the set $|x|\geq \epsilon _0$ has the value $f(x)=|x|$ (note that the absolute value function $|x|$ is nowadays one of the most popular functions in artificial neural networks). We can easily construct such functions with the additional requirement: $f(x)$ has exactly one critical point $x=0$ which is a global minimum point. Then the derivative $f'(x)$ satisfies: $f'(x)=1$ if $x\geq \epsilon _0$ and $f'(x)=-1$ if $x\leq -\epsilon _0$. Let $\delta _0$ be an arbitrary positive number $>2\epsilon _0$, and $z_0$ be an arbitrary number satisfying $\epsilon _0<z_0<\delta _0-\epsilon _0$. Then the sequence in (\ref{Equation1})  is periodic. In fact, since $z_0>\epsilon _0$ we have $f'(z_0)=1$, and hence  $z_1=z_0-\delta _0f'(z_0)=z_0-\delta _0$. By the conditions on $\epsilon _0$ and $\delta _0$, we have $z_1=z_0-\delta _0<-\epsilon _0$ and hence $f'(z_1)=-1$. Therefore, $z_2=z_1-\delta _0f'(z_1)$ $=$ $(z_0-\delta _0)- \delta _0 \times (-1)=z_0$. Thus the sequence $\{z_n\}$ is periodic, and the cluster points of it are $z_0$ and $z_1$, neither of them is the unique critical point $0$ of $f$.  

\begin{figure}
\begin{center}
{\includegraphics[width=0.8\textwidth]{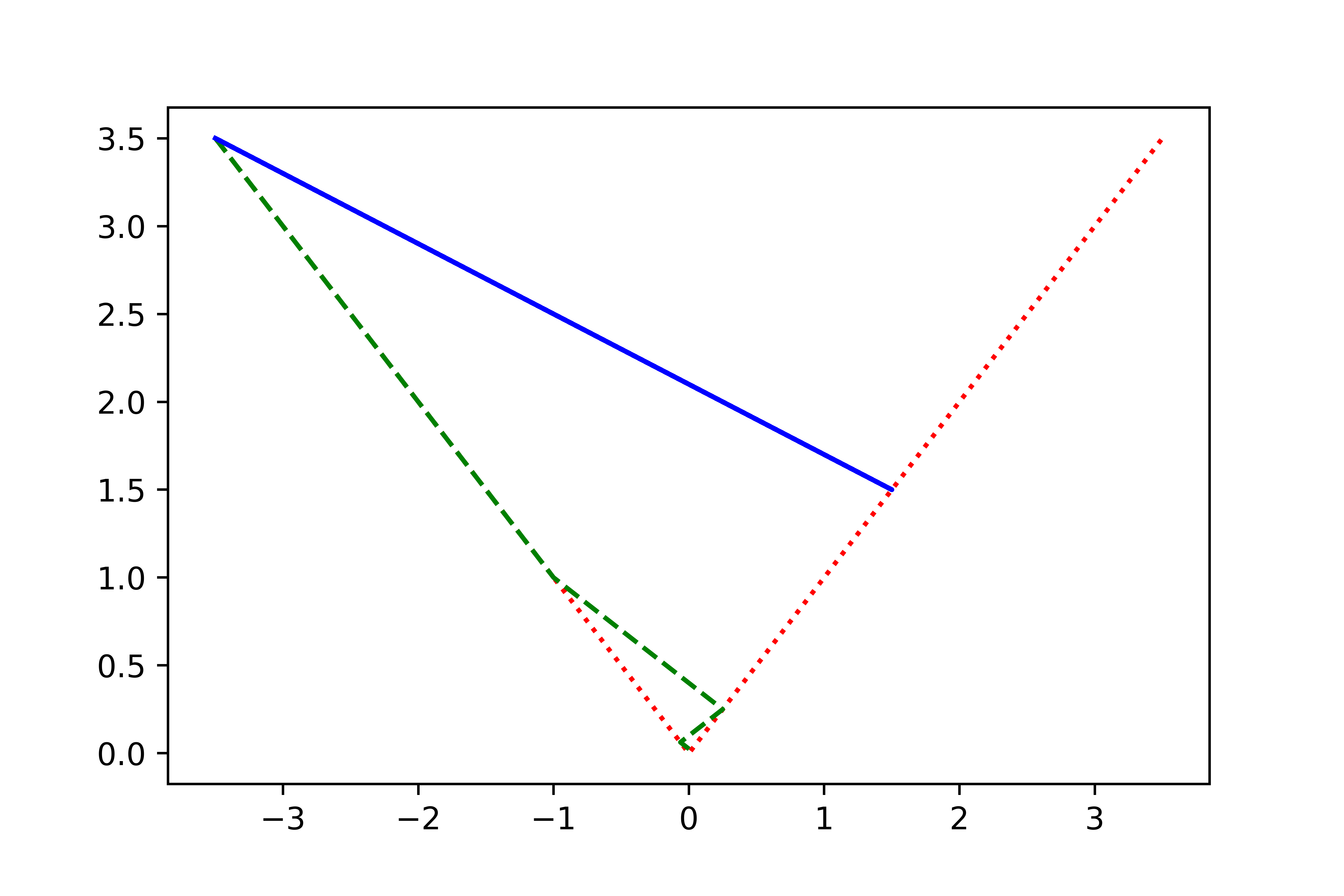}}
\end{center}
\caption{The behaviours of Standard GD and Backtracking GD for Example \ref{Example1}. Red dotted line: the function $f(x)=|x|$; blue solid line: Standard GD gets stuck when going back and forth between two points; green dashed line: Backtracking GD - which always makes sure $f(x-\delta (f,\delta _0,x)\nabla f(x))\leq f(x)$  - works properly to reach the local minimum.} 
\label{fig1}
\end{figure}
\label{Example1}\end{example}

Concerning the issue of saddle points, we have the following very strong result for functions in class $C^{1,1}_L$ (\cite{lee-simchowitz-jordan-recht, panageas-piliouras}).
\begin{theorem}
Let $f$ be in $C^{1,1}_L$ and $\delta < 1/L$. Then there exists a set $E\subset \mathbb{R}^k$ of Lebesgue measure $0$ so that if $x_0\in \mathbb{R}^k\backslash E$, then $\{x_n\}$ in Standard GD, if converges, will not converge to a saddle point.  
\label{TheoremGradientDescentSaddlePoint}\end{theorem} 
The main idea is that then the map $x\mapsto x-\delta \nabla f(x)$ is a diffeomorphism, and hence we can use the Stable-Center manifold theorem in dynamical systems (cited as Theorem 4.4 in \cite{lee-simchowitz-jordan-recht}). For to deal with the case where the set of critical points of the function is uncountable, the new idea in \cite{panageas-piliouras} is to use Lindel\"off lemma that any open cover of an open subset of $\mathbb{R}^m$ has a countable subcover. However, for convergence of $\{z_n\}$, one has to use Theorem \ref{TheoremArmijoGeneralisation}, and needs to assume more, as seen from Example \ref{LemmaHolderContinuity} above. Therefore, Proposition 4.9 in \cite{lee-simchowitz-jordan-recht} is not valid as stated. 

Theorem \ref{Theorem2}, while has weaker conclusion than that of Theorem \ref{TheoremGradientDescentSaddlePoint}, can be applied to all functions and hence can also be used to justify for the fact that GD in most of the case will contain only minima as cluster points. The next example shows that the idea of using dynamical systems, as in \cite{lee-simchowitz-jordan-recht}, at current cannot be used for Backtracking GD. 
 \begin{example}
Let $\delta _0>0$ be one solution of the equation $p(t)=6t^2-6t+1=0$ (we can see that this equation has one positive solution by observing that $p(1/2)=-1/2<0$ and $p(+\infty )=+\infty$). Let $f:\mathbb{R}\rightarrow \mathbb{R}$ be the function $f(x)=x^3$. Then the map $g(x)=x-\delta (f,\delta _0,x)f'(x)$ is not continuous at $x_0=1$. In fact, the choice of $\delta _0$ implies that 
\begin{eqnarray*}
f(1-\delta _0f'(1))-f(1)=-\delta _0|f'(1)|^2/2,
\end{eqnarray*}
hence $\delta (f,\delta _0,1)=\delta _0$. Since $f'(x)$ is not identically $0$ in a neighbourhood of $1$ and the function $\delta (f,\delta _0,x)$ takes values in a discrete set, if $g(x)$ were continuous at $1$ as a function in $x$, we would have $\delta (f,\delta _0,x)=\delta _0$ for all $x$ sufficiently close to $1$. The latter means that for $x$ sufficiently close to $1$, we have 
\begin{eqnarray*}
(x-3\delta _0x^2)^3-x^3+9\delta _0x^4/2 \leq 0.
\end{eqnarray*}
Dividing by $x^3$, we see that $h(x)=(1-3\delta _0x)^3-1+9\delta _0x/2\leq 0$ for all $x$ closes to $1$. However, it can be checked that for the choice of $\delta _0$, $h'(x)= 9\delta _0(3\delta _0x-1)^2+9\delta _0/2 >0$, and hence there must be a sequence $x_n\rightarrow 1$ so that $h(x_n)>h(1)=0$ for all $n$. Thus, for this sequence, $g(x_n)$ does not converge to $g(1)$, as claimed.
\label{Example2}\end{example}

A key point in the proofs of these papers, as well as of the paper  \cite{absil-mahony-andrews}, is that some estimates on the convergence rate for functions in $C^{1,1}_L$ can be explicitly obtained. In contrast, our proof of Theorem \ref{Theorem1} is very indirect, since such estimates are not available for general $C^1$ functions. However, the proof of part 1 of Theorem \ref{Theorem1} suggests that a  reasonable criterion when working with a general function $f\in C^1$ is to stop the iteration when $\delta _n||\nabla f(x_n)||$ is small enough.    

There are other variants of GD which are regarded as state-of-the-art algorithms in DNN such as MMT and NAG (mentioned in Subsection \ref{SectionBacktrackingMMTNAG}), Adam, Adagrad, Adadelta, and RMSProp  (see an overview in \cite{ruder}). Some of these variants (such as Adagrad and Adadelta) allow choosing learning rates $\delta _n$ to decrease to $0$ (inspired by Stochastic GD, see next paragraph) in some complicated manners which depend on the values of gradients at the previous points  $x_0,\ldots ,x_{n-1}$. However, as far as we know, convergence for such methods are not yet available beyond the usual setting such as in Theorem \ref{TheoremArmijoGeneralisation}.  

Stochastic GD is the default method used to justify the use of GD in DNN, which goes back to Robbins and Monro, see \cite{bottou-etal}. The most common version of it is to assume that we have a fixed cost function $F$ (as in the deterministic case), but we replace the gradient $\nabla _{\kappa}F(\kappa _n)$ by a random vector $g_n(\kappa _n)$ (here the random variables are points in the dataset, see also Inexact GD), and then show the convergence in probability of the sequence of values $F(\kappa _n)$ and of gradients $\nabla _{\kappa}F(\kappa _n)$ to $0$ (in application the random vector $g_n(\kappa _n)$ will be $\nabla _{\kappa }F_{I_n}(\kappa _n)$). However, the assumptions for these convergence results (for $F(\kappa _n)$ and $\nabla _{\kappa}F(\kappa _n)$) to be valid still require those in the usual setting as in Theorem \ref{TheoremArmijoGeneralisation}. In the case where there is noise, the  following additional conditions on the learning rates are needed \cite{robbins-monro}:
\begin{equation}
\sum _{n\geq 1}\delta _n=\infty,~ \sum _{n\geq 1}\delta _n^2<\infty .
\label{EquationStochasticGD}\end{equation} 
However, in Standard GD, which is the most common used version in DNN, all the learning rates are the same and hence condition  $\sum _{n\geq 1}\delta _n^2<\infty$ is violated. Moreover, showing that the gradients $\nabla _{\kappa}F(\kappa _n)$ converge to $0$ is far from proving the convergence of $\kappa _n$ itself. 

When using GD in DNN, even when the underlying function $F$ is in $C^{1,1}_L$, it may be difficult to obtain a good lower bound estimate for the Lipschitz constant $L$, since these functions can have thousands of variables. Hence it can be difficult to obtain a good choice for the learning rate $\delta _0$. The common practice in DNN is to manually fine-tune learning rates \cite{nielsen}: trial and error, do experiments and then observe and modify learning rates until achieving an acceptable descent of the cost (or loss) function. However, this practice is very time-consuming (especially when working with large dataset and/or complicated architectures) and depending too much on the researcher's experience. In contrast, Backtracking GD is automatic.     

{\bf Wolfe's method.} Since Wolfe's method is very close to Backtracking GD, we provide a more detailed comparison for it here. There are some abstract conditions on functions with only directional derivatives (see the definition about serious steps on page 228 in \cite{wolfe}), under which convergence results can be proven. Note, however, that the descent process proposed by Wolfe in \cite{wolfe} is more complicated than Backtracking GD, since at each iteration the learning rate $\delta _n$ - denoted $t_n$ in his paper - and hence the point $z_{n+1}$, is to be chosen with respect to some of 5 choices listed in Definition on page 228 in that paper, not the unique one based on $||\nabla f(z_n)||$ alone as usually done in practice and considered in the current paper. We mention here in particular conditions iii) and iv). Condition iii) is exactly Armijo's condition (\ref{EquationArmijoCondition}). Condition iv) is that for a fixed $c_2>0$ there is $\xi _n$ between $z_n$ and $z_{n+1}$ so that $<\nabla f(\xi _n),v_n>\geq c_2<\nabla f(z_n),v_n>$ {\bf and} $f$ is non-increasing from $z_n$ to $\xi _n$.  

Extracted from  the above two conditions iii) and iv) in Wolfe's method are the following two conditions, usually called Wolfe's conditions
\begin{eqnarray*}
f(z_{n}-\delta _nv_n)-f(z_n)&\leq&-c_1 \delta _n<\nabla f(z_n), v_n>,\\
<\nabla f(z_{n}-\delta _nv_n),v_n > &\leq& c_2 <\nabla f(z_n),v_n>,
\end{eqnarray*}
for some fixed constants $1>c_2>c_1>0$. The first condition is exactly Armijo's theorem (condition iii) in Wolfe's paper). The second condition is only a half of condition iv) in Wolfe's method. It has been shown that if $f$ is a $C^1$ function which is {\bf bounded from below}, then a positive $\delta _n$ can be chosen to satisfy these Wolfe's conditions. Moreover, if $f$ is in $C^{1,1}_L$ and $v_n$ satisfies condition i) in Inexact GD,  then a result by G.  Zoutendijk shows the convergence of $\nabla f(z_n)$ to $0$. For more details, the readers can consult \cite{nocedal-wright}. Therefore, under the assumptions mentioned above (that is $f$ is in $C^{1,1}_L$, bounded from below, and $v_n$ satisfies condition i) in Inexact GD), by combining with Remark \ref{ExampleBestPossible}, we can prove conclusions of Theorem \ref{Theorem1} with Armijo's rule replaced by Wolfe's conditions.  Hence, we can see that the scope of application of Backtracking GD is wider than that of using Wolfe's conditions. Moreover, since Backtracking GD needs only Armijo's condition, while Wolfe's conditions require more, in general the learning rates determined by Backtracking GD will be not smaller than that determined from Wolfe's conditions, and hence intuitively will be better for convergence.  Recently, there is an implementation of Wolfe's conditions in DNN, see \cite{mahrsereci-hennig}. 

\section{Proofs of main theoretical results and further examples}\label{Section2}

The next two lemmas  are key to our main results. Both lemmas are probably well-known to the experts, but because of lack of proper references and because of completeness, we include their proofs here. 

\begin{lemma} Let $f$ be a $C^1$ function. Then

1) For any $x\in \mathbb{R}^m$, there is a positive integer $n_0$ for which the following is satisfied:
\begin{eqnarray*}
f(x-{\beta ^{n_0}\delta _0}\nabla f(x))-f(x)\leq - {\beta ^{n_0}\delta _0\alpha}||\nabla f(x)||^2. 
\end{eqnarray*}

2) For any compact subset $K$ of $\mathbb{R}^m$ with $\inf _{x\in K}||\nabla f(x)||>0$, we have 
\begin{eqnarray*}
\inf _{x\in K}\delta (f,\delta _0,x) >0. 
\end{eqnarray*}

\label{Lemma2}\end{lemma}
\begin{proof}
We will give the proof for $\alpha =1/2$. The other cases can be treated similarly. 

1) This is a simple consequence of Taylor's expansion for a continuously differentiable multivariable function. Below is the detail. We define a function $g:\mathbb{R}\rightarrow \mathbb{R}$ by the formula $g(t)=f(x-t\delta _0\nabla f(x))$. Then $g$ is continuously differentiable, and by the chain rule $g'(t)=-\delta _0\nabla f(x-t\delta _0\nabla _0f(x)).\nabla f(x)$. Here we use the dot product between two vectors $\nabla f(x-t\delta _0\nabla _0f(x))$ and $\nabla f(x)$ in $\mathbb{R}^m$: if  $u=(u_1,u_2,\ldots ,u_n)$ and $v=(v_1,v_2,\ldots ,v_n)$ then $u.v=u_1v_1+\ldots +u_nv_n$. The Fundamental Theorem of Calculus: $g(1)-g(0)=\int _{t}^1g'(s)ds$ can be explicitly written in this case as follows
\begin{equation}
f(x-t\delta _0\nabla f(x))-f(x)=-\delta _0\int _0^t\nabla f(x-s\delta _0\nabla _0f(x)).\nabla f(x)ds. 
\label{Equation4}\end{equation} 
If $\nabla f(x)=0$, then we can choose simply $n_0=0$. Hence, we can assume for {\bf the remaining} of the proof that $\nabla f(x)\not= 0$, and hence $||\nabla f(x)||>0$. Because $f$ is continuously differentiable, for the positive number $\epsilon _0= 1/2 ||\nabla f(x)||^2$, there is a $\gamma _0>0$  so that if $t_0>0$ is such that $t_0||\delta _0\nabla f(x)||<\gamma _0$ then 
\begin{eqnarray*}
||\nabla f(x-s\delta _0\nabla _0f(x))-\nabla f(x)||\leq \epsilon _0=1/2 ||\nabla f(x)||,
\end{eqnarray*}
for all $0\leq s\leq t_0$. By the triangular inequality, we have the following simple estimate for the integrand in (\ref{Equation3})
\begin{eqnarray*}
\nabla f(x-s\delta _0\nabla _0f(x)).\nabla f(x)&=&(\nabla f(x-s\delta _0\nabla _0f(x))-\nabla f(x)).\nabla f(x)+||\nabla f(x)||^2\\
&\geq&||\nabla f(x)||^2-||\nabla f(x-s\delta _0\nabla _0f(x))-\nabla f(x)||\times ||\nabla f(x)||.
\end{eqnarray*}
Hence, with this choice of $t_0$, for all $0\leq t\leq t_0$ we have from (\ref{Equation3}) that
\begin{eqnarray*}
f(x-t\delta _0\nabla f(x))-f(x)\leq -\frac{t\delta _0}{2}||\nabla f(x)||^2.
\end{eqnarray*}
We can always find a positive integer $n_0$ so that $t=\beta ^{n_0}$ satisfies the needed condition, and for this choice of $t$ we obtain the conclusion of the lemma. 

2) In the proof of 1), we choose $\epsilon _0=\inf _{x\in K}||\nabla f(x)||/2$. Then by the assumption in (1), we have that $\epsilon _0>0$. Since $\nabla f(x)$ is uniformly continuous on compact subsets of $\mathbb{R}^m$, there is a $\gamma _0>0$ for which whenever 
\begin{eqnarray*}
\sup _{x\in K}t_0||\delta _0\nabla f(x)||\leq \gamma _0,
\end{eqnarray*}
then for all $0\leq s\leq t_0$
\begin{eqnarray*}
\sup _{x\in K}||\nabla f(x-s\delta _0\nabla _0f(x))-\nabla f(x)||\leq \epsilon _0=1/2 \inf _{x\in K}||\nabla f(x)||.
\end{eqnarray*}
We can then repeat the remaining of the proof of 1). 
\end{proof}
 
 The next lemma concerns a simple property of trigonometric functions. We recall that the function $\arccos :[0,1]\rightarrow [0,\pi /2]$ is defined so that $\arccos (u)=v$ iff $u=\cos (v)$. For $x=(x_1,\ldots ,x_m), y=(y_1,\ldots ,y_m)\in \mathbb{R}^m$, we define
 \begin{equation}
 dist (x,y)=\arccos (\frac{|1+\sum _{i=1}^mx_iy_i|}{\sqrt{1+\sum _{i=1}^mx_i^2}\sqrt{1+\sum _{i=1}^my_i^2}}). 
 \label{EquationDistance}\end{equation}
 Recall that by Cauchy-Schwartz inequality we always have   
 \begin{eqnarray*}
 0\leq \frac{|1+\sum _{i=1}^mx_iy_i|}{\sqrt{1+\sum _{i=1}^mx_i^2}\sqrt{1+\sum _{i=1}^my_i^2}}\leq 1,
 \end{eqnarray*}
 and hence the function $dist (x,y)$ is well-defined. As used throughout the whole paper, $||x-y||=\sqrt{\sum _{i=1}^m|x_i-y_i|^2}$ is the usual Euclidean distance on $\mathbb{R}^m$.  We have the following result. 
 \begin{lemma}
 There is a constant $C>0$ so that for all $x,y\in \mathbb{R}^m$ we have
 \begin{eqnarray*}
 C||x-y||\geq dist (x,y). 
 \end{eqnarray*}
    \label{Lemma3}\end{lemma}
 \begin{proof}
  We will choose $C\geq \pi$. Therefore, if $||x-y||\geq 1$ then there is nothing to prove, since the range of the arccos function is $[0,\pi /2]$. 
  
  For $0\leq \epsilon \leq 1$, we define 
 \begin{eqnarray*}
 h(\epsilon )=\inf _{x,y\in \mathbb{R}^m:~||x-y||\leq \epsilon} \frac{|1+\sum _{i=1}^mx_iy_i|}{\sqrt{1+\sum _{i=1}^mx_i^2}\sqrt{1+\sum _{i=1}^my_i^2}}.
 \end{eqnarray*}
 
 By Cauchy-Schwartz inequality $1+(a+b)/2\geq \sqrt{1+a}\sqrt{1+b}$ (which the readers can easily check by squaring two sides and simplifying), for $a=\sum _{i=1}^mx_i^2$ and $b=\sum _{i=1}^my_i^2$, we have
 \begin{eqnarray*}
\frac{1}{\sqrt{1+\sum _{i=1}^mx_i^2}\sqrt{1+\sum _{i=1}^my_i^2}}\geq \frac{1}{1+(\sum _{i=1}^mx_i^2+\sum _{i=1}^my_i^2)/2}. 
 \end{eqnarray*}
  
By assumption, $0\leq \epsilon \leq 1$, and hence if $||x-y||\leq \epsilon\leq 1$ we have by Cauchy-Schwarz inequality
\begin{eqnarray*}
1+\sum _{i=1}^mx_iy_i&=&1+\sum _{i=1}^mx_i^2+\sum _{i=1}^mx_i(y_i-x_i)\\
&\geq& 1+\sum _{i=1}^mx_i^2-\sqrt{\sum _{i=1}^mx_i^2}\sqrt{\sum _{i=1}^m(x_i-y_i)^2}\\
&\geq& 1+ \sum _{i=1}^mx_i^2-\sqrt{\sum _{i=1}^mx_i^2}\geq 0. 
\end{eqnarray*} 
  
 Therefore, under the same assumption on $x,y$ and $\epsilon$:
 \begin{eqnarray*}
 \frac{|1+\sum _{i=1}^mx_iy_i|}{\sqrt{1+\sum _{i=1}^mx_i^2}\sqrt{1+\sum _{i=1}^my_i^2}}&=&\frac{1+\sum _{i=1}^mx_iy_i}{\sqrt{1+\sum _{i=1}^mx_i^2}\sqrt{1+\sum _{i=1}^my_i^2}}\\
 &\geq& \frac{1+\sum _{i=1}^mx_iy_i}{1+(\sum _{i=1}^mx_i^2+\sum _{i=1}^my_i^2)/2}\\
 &=&1-\frac{1}{2}\frac{||x-y||^2}{1+(\sum _{i=1}^mx_i^2+\sum _{i=1^my_i^2})/2}\\
 &\geq&1-\frac{1}{2}||x-y||^2\geq 1-\frac{1}{2}\epsilon ^2.
 \end{eqnarray*}
 Therefore, $h(\epsilon )\geq 1-\epsilon ^2/2$. Using that the arccos function is decreasing on its domain of definition, we then have for $||x-y||\leq \epsilon \leq 1$: 
\begin{eqnarray*}
dist (x,y)&=& \arccos (\frac{|1+\sum _{i=1}^mx_iy_i|}{\sqrt{1+\sum _{i=1}^mx_i^2}\sqrt{1+\sum _{i=1}^my_i^2}})\\
&\leq& \arccos (h(\epsilon )) \leq \arccos (1-\epsilon ^2/2). 
\end{eqnarray*} 
Now, using the classical inequality that $\cos (\pi \epsilon )\leq 1-\epsilon ^2/2$ if $\epsilon\leq \epsilon _0 $ for a small enough $\epsilon _0>0$ (explicitly determined), we have that provided $||x-y||\leq \epsilon _0$ then
\begin{eqnarray*}
dist (x,y)\leq \pi ||x-y||.
\end{eqnarray*}    
Therefore, if we choose $C=\pi /\epsilon _0$, we obtain for all $x,y\in \mathbb{R}^m$: 
\begin{eqnarray*}
dist(x,y)\leq C||x-y||,
\end{eqnarray*}
 as desired.
\end{proof}
 
\begin{definition}[Compact metric spaces and real projective spaces] We will need the notation of a compact metric space and in particular the real projective space $\mathbb{P}\mathbb{R}^m$ which we now briefly recall for the readers' convenience. A {\bf metric space} $(X,d)$ is a set $X$ equipped with a distance $d:X\times X\rightarrow [0,\infty )$, with the following three properties: i) $d(x,y)=0$ iff $x=y$, ii) (Symmetry) $d(x,y)=d(y,x)$, and iii) (Triangle inequality) $d(x,y)+d(y,z)\geq d(x,z)$ for all $x,y,z\in X$. A sequence $\{x_n\}$ is said to converge to $x$ in $(X,d)$ if $\lim _{n\rightarrow\infty}d(x_n,x)=0$. A  metric space $(X,d)$ is {\bf compact} if any sequence $\{x_n\}$ has a convergent subsequence. Note that the usual Euclidean space $(X=\mathbb{R}^m,d=||.||)$ is a metric space but is {\bf not} compact. However, we can define a compact metric space $(\mathbb{P}\mathbb{R}^m,d)$ - called the {\bf real projective space} of dimension $m$ - with the following three properties: i) $\mathbb{P}\mathbb{R}^m$ contains $\mathbb{R}^m$ as a set, ii) For $x,y\in \mathbb{R}^m$ the distance $d(x,y)$ is exactly the function $dist (x,y)$ in (\ref{EquationDistance}), and iii) if $\{x_n\}\subset \mathbb{R}^m$ converges  in $(\mathbb{P}\mathbb{R}^m,d)$ to a point $z\in \mathbb{P}\mathbb{R}^m\backslash \mathbb{R}^m$, then $\lim _{n\rightarrow \infty}||x_n||=\infty$.    
\label{DefinitionRealProjectiveSpace}\end{definition} 
 
We are now ready to prove the main results of this paper. 

\begin{proof}[Proof of Theorem \ref{Theorem1}] For simplicity, we give the proof for $\alpha =1/2$ only. The other cases can be treated similarly. 

1) By construction, we have 
\begin{eqnarray*}
f(z_{n+1})-f(z_n)\leq -||\nabla f(z_n)||\times ||z_{n+1}-z_n||/2,
\end{eqnarray*}
and hence (by multiplying both sides with $-1$) for all $n$:
\begin{eqnarray*}
f(z_{n})-f(z_{n+1})\geq ||\nabla f(z_n)||\times ||z_{n+1}-z_n||/2.
\end{eqnarray*}
 
Since $\{f(z_n)\}$ is a decreasing sequence, we have two cases to consider. 

Case 1: $\lim _{n\rightarrow\infty}f(z_n)=-\infty$. In this case, it follows easily (since $f$ is a continuous function, and hence is bounded on any compact set) that $\lim _{n\rightarrow\infty}||z_n||=\infty$. 

Case 2: $\{f(z_n)\}$ is bounded. Then
\begin{eqnarray*}
\lim _{n\rightarrow\infty}(f(z_n)-f(z_{n+1}))=0. 
\end{eqnarray*}

Fix $\epsilon >0$ a small number. We partition $\{0,1,2,3,\ldots \}$ into $2$ sets: 
\begin{eqnarray*}
A(\epsilon )=\{n:~\delta (f,\delta _0,z_n)||\nabla f(z_n)||\leq \epsilon \};~~B(\epsilon )&=&\{0,1,2,\ldots \}\backslash A(\epsilon ).
\end{eqnarray*}

For $n\in A(\epsilon )$, we have from (\ref{Equation3}) that
\begin{eqnarray*}
||z_{n+1}-z_n||=\delta (f,\delta _0,z_n)||\nabla f(z_n)||\leq \epsilon .
\end{eqnarray*}

For $n\in B(\epsilon )$ we have
\begin{eqnarray*}
||\nabla f(z_n)||\geq \frac{\epsilon}{\delta (f,\delta _0,z_n)}\geq \frac{\epsilon}{\delta _0}. 
\end{eqnarray*} 

Now we can choose $n(\epsilon )>0$ so that for $n\geq n(\epsilon )$ then
\begin{eqnarray*}
0\leq f(z_n)-f(z_{n+1})\leq \epsilon ^2. 
\end{eqnarray*}
Combining the above inequalities, we obtain, for $n\geq n(\epsilon )$ and $n\in B(\epsilon )$
\begin{eqnarray*}
\epsilon ^2&\geq&f(z_n)-f(z_{n+1})\geq \delta (f,\delta _0,z_n)||\nabla f(z_n)||^2/2\\
&=&||z_{n+1}-z_n|| \times ||\nabla f(z_n)||/2\\
&\geq &\frac{\epsilon}{2\delta _0}||z_{n+1}-z_n||. 
\end{eqnarray*}
Therefore, $2\delta _0\epsilon \geq ||z_{n+1}-z_n||$ if $n\in B(\epsilon )$ and $n\geq n(\epsilon )$. Therefore, for all $n\geq n(\epsilon )$ we have $||z_{n+1}-z_n||\leq \max \{\epsilon ,2\delta _0\epsilon\}$. Since $\epsilon >0$ is arbitrary, it follows that in Case 2 we have 
\begin{eqnarray*}
\lim _{n\rightarrow\infty}||z_{n+1}-z_n||=0. 
\end{eqnarray*}

2) We let $(\mathbb{P}\mathbb{R}^m,d)$ be the real projective space of dimension $m$ which we introduced in the front of the proof of Theorem \ref{Theorem1}. Let $\{z_n\}$ be the sequence from (\ref{Equation3}), with an arbitrary initial point $z_0\in \mathbb{R}^m$ and an arbitrary choice of $\delta _0$. 

We have two cases to consider. 

Case 1: $\lim _{n\rightarrow\infty}f(z_n)=-\infty$. In this case, we then have $\lim _{n\rightarrow\infty}||z_n||=\infty$. 

Case 2: The remaining case where $\{f(z_n)\}$ is bounded. In this case, by part 1 above we have $\lim _{n\rightarrow\infty}||z_{n+1}-z_n||=0$. By Lemma \ref{Lemma3} and the properties of the real projective space mentioned in the front of the proof of Theorem \ref{Theorem1}, we have $d(z_{n+1},z_n)\leq C||z_{n+1}-z_n||$ for all $n$, where $C>0$ is a constant. In particular, we also have $\lim _{n\rightarrow\infty}d(z_{n+1},z_n)=0$. 

Let $D$ be the cluster set of the sequence $\{z_n\}$ in the usual Euclidean space $(\mathbb{R}^m, ||.||)$, and let $D'$ be the cluster set of the sequence $\{z_n\}$ in the real projective space  $(\mathbb{P}\mathbb{R}^m,d)$. While the metrics $||.||$ and $d$ are different on $\mathbb{R}^m$, they induce the same topology on $\mathbb{R}^m$. Therefore, from elementary point set  topology, we obtain that $D'$ is equal to the closure $\overline{D}$ of $D$ in $(\mathbb{P}\mathbb{R}^m,d)$, and $D=D'\cap \mathbb{R}^m$.

Because in Case 2 we showed above that $\lim _{n\rightarrow\infty}d(z_{n+1},z_n)=0$, we can apply results in \cite{asic-adamovic} to the compact metric space $(\mathbb{P}\mathbb{R}^m,d)$ to obtain that $D'$ is connected. Since $D$ must be contained in the set of critical points of $f$, it is at most countable by the assumption. Then from elementary point set topology we have the following conclusion: 

i) Either $D=\emptyset$, thus all cluster points of $\{z_n\}$ are contained in $\mathbb{P}\mathbb{R}^m\backslash \mathbb{R}^m$, and hence $\lim _{n\rightarrow\infty}||z_n||=\infty$,

or 

ii) $D=D'= $ 1 point, that is $\{z_n\}$ converges to a unique point $z_{\infty}\in \mathbb{R}^m$.

3) The proof is similar to that of part 2 above. 
\end{proof}

Next we give the proof of Theorem \ref{Theorem2}. 
\begin{proof}[Proof of Theorem \ref{Theorem2}] As above, we will treat only the case $\alpha =1/2$ here.  We may assume that $z_{\infty}=0$ and $f(0)=0$. Since $z_{\infty}=0$ is a critical point of $f$, we have $\nabla f(0)=0$. Hence, for $z\in B(0,\epsilon _0)$, where $\epsilon _0>0$ small enough, and for $0<\delta <\delta _0$, we get by Taylor's expansion
\begin{eqnarray*}
\nabla f(z)&=&\nabla f(0).z+\int _0^1\nabla ^2f(tz).zdt\\
&=&\int _0^1\nabla ^2f(0).zdt+\int _0^1[\nabla ^2f(tz)-\nabla ^2f(0)].zdt\\
&=&\nabla ^2f(0).z+\gamma _1(z).||z||,
\end{eqnarray*}
where $\lim _{z\rightarrow 0}\gamma _1(z)=0$. 

Using the same argument we obtain
\begin{eqnarray*}
f(z-\delta \nabla f(z))&=&f(0)+\nabla f(0).(z-\delta \nabla f(z))\\
&&+\int _0^1\nabla ^2f(t(z-\delta \nabla f(z))).(z-\delta \nabla f(z)).(z-\delta \nabla f(z))dt\\
&=&\nabla ^2f(0).(z-\nabla ^2f(0).z).(z-\nabla ^2f(z_0).z)+\gamma _2(z)||z||^2,
\end{eqnarray*}
where $\lim _{z\rightarrow 0}\gamma _2(z)=0$. 

Since $\nabla ^2f(0)$ is a symmetric real square matrix (hence is diagonalisable) and with at least one negative eigenvalue, by Silvester's law of inertia we may assume that in $B(0,\epsilon )$ the quadratic form $\nabla ^2f(0).z.z$, where $z=(z_1,\ldots ,z_k)$ has the form (where $z_1,\ldots ,z_j$ correspond to eigenvectors of positive eigenvalues, $z_{j+1},\ldots ,z_m$ correspond to eigenvectors of zero eigenvalue, and $z_{m+1},\ldots ,z_k$ correspond to eigenvectors of positive eigenvalues of $\nabla ^2f(0)$, and hence we can - after applying a linear change of coordinates - assume that the eigenvalues of $\nabla ^2f(0)$ are $\pm 1$ and $0$)
\begin{eqnarray*}
\nabla ^2f(0).z.z=(z_1^2+\ldots +z_j^2)-(z_{m+1}^2+\ldots +z_k^2),
\end{eqnarray*}
where $m\leq k-1$. From the above calculations, we obtain (noting that $\nabla ^2f(0).z=(z_1,\ldots ,z_m,-z_{m+1},\ldots ,-z_k)$)
\begin{eqnarray*}
f(z-\delta \nabla f(z))&=&\nabla ^2f(0).(z-\nabla ^2f(0).z).(z-\nabla ^2f(0).z)+\gamma _2(z)||z||^2\\
&=&-4(z_{m+1}^2+\ldots +z_k^2)+\gamma _2(z)(z_1^2+\ldots +z_k^2).
\end{eqnarray*}
Therefore, provided $z=(z_1,\ldots ,z_k)\in B(0,\epsilon _0)$ and $4(z_{m+1}^2+\ldots +z_k^2)>|\gamma _2(z)|(z_1^2+\ldots +z_k^2)$, then for the sequence $\{z_n\}$ in Equation (\ref{Equation3}) with the initial value $z_0=z$ we have $f(z_n)\leq f(z_1)<0$ for all $n\geq 1$. Hence $\{z_n\}$ does not contain any subsequence converging to $z_{\infty}=0$. Therefore, such a $z$ belongs to $\mathcal{D}(z_{\infty})$. 

If we define, for any $0<\epsilon \leq \epsilon _0$, the number $\gamma _{\epsilon }=\sup _{z\in B(0,\epsilon )}|\gamma _2(z)|$, then $\lim _{\epsilon \rightarrow 0}\gamma _{\epsilon }=0$. Moreover, the open set $$U(z_{\infty},\epsilon )=\{z=(z_1,\ldots ,z_k)\in B(z_{\infty},\epsilon ):~4(z_{m+1}^2+\ldots +z_k^2)>\gamma _{\epsilon}(z_1^2+\ldots +z_k^2)\}$$ belongs to $\mathcal{D}(z_{\infty})$, and   
\begin{eqnarray*}
\lim _{\epsilon\rightarrow 0}\frac{Vol (U(z_{\infty}, \epsilon ))}{Vol (B(z_{\infty}, \epsilon ))}=1. 
\end{eqnarray*}
Therefore, the proof of the theorem is completed. 
\end{proof}

\section{Experimental results}\label{Section3}
 In this section we illustrate the effectiveness of the new methods with experiments on the CIFAR10 and CIFAR100 datasets, using various state-of-the-art DNN models Resnet18 (\cite{He}), MobileNetV2 (\cite{Sandler}), SENet(\cite{Hu}), PreActResnet18(\cite{He2}) and Densenet121(\cite{Huang}) showing that it is on par with current state-of-the-art methods such as MMT, NAG, Adagrad, Adadelta, RMSProp, Adam and Adamax, see \cite{ruder}. 
 
Before presenting the details of the  experimental results, we present some remarks about the actual implementation of our methods in DNN, to deal with the randomness arising from the practice of using mini-batches.      
  
\subsection{Some remarks about implementation of GD in DNN} In this section we mention some issues one faces when applying GD (in particular, Backtracking GD) to the mini-batch practice in DNN, and also the actual implementations of mini-batc backtracking methods MBT-GD, MBT-MMT and MBT-NAG.  
 
\subsubsection{Rescaling of learning rates}\label{SectionRescaling} Since the cost functions obtained from different mini-batches are not exact even though they may be close, using GD iterations many times can accumulate errors and lead to explosion of errors. If we use the backtracking method for a mini-batch, the obtained learning rate is only optimal in case we continue the training with that mini-batch. Even if we take the average of the learning rates from many batches, using directly this value in GD can cause a lot of noise from covariance between batches. This phenomenon has been observed in practice, and the method of rescaling of learning rates has been proposed to prevent it, in both Standard GD \cite{Krizhevsky} and Wolfe's method \cite{mahrsereci-hennig}. The main idea in these papers is that we should rescale learning rates depending on the size of mini-batches, and a popular choice is to use a linear dependence. Roughly speaking, if $\delta $ is a theoretically good learning rate for the cost function of full batch, and the ratio between the size of full batch and that of a mini-batch is $\rho=k/N$ (where $N=$ the size of full batch and $k=$ the size of a mini-batch), then this popular choice suggests to use instead the learning rate $\delta /\rho$. 

Here, we propose a new rescaling scheme that the bigger learning rate $\delta /\sqrt{\rho}$ should also work in practice. Our justification comes from the stochastic gradient updating with Gaussian noise (\cite{Jastrzebski}):
\begin{equation}
z_{n+1}=z_n-\delta _n (\nabla f(z_n) + \frac{r}{\sqrt{N}}),
\label{EquationGradientDescentWithNoise}\end{equation}
with $r$ is a zero mean Gaussian random variable with covariance of $\nabla f(z_n)$ from whole population of samples. The updating term $\delta _n (\nabla f(z_n) + \frac{r}{\sqrt{N}})$ bears a random variable noise $\delta _n \frac{r}{\sqrt{N}}$. In case of mini-batching, this noise is $\delta _n \frac{r}{\sqrt{k}}$. If we want to maintain the noise level of mini-batch similar to that of full batch, we can simply rescale $\delta_n$ by a factor $\sqrt{\frac{k}{N}}$. One can notice that using $\rho=k/N$ in (\cite{Jastrzebski}) has the effect of trying to make the training loss/ accuracy of using small mini-batch the same as that of using larger mini-batch or full batch after each epoch. This is not essential because what we really need is fast convergence and high accuracy for stochastic optimising, not full batch imitation. A larger learning rate such as $\rho = \sqrt{\frac{k}{N}}$ can help to accelerate the convergence while still keep low noise level. \cite{Smith2} demonstrated that networks trained with large learning rates use fewer iterations, improve regularization and obtain higher accuracies.

We have checked with experiments, see below, that this new rescaling scheme works well in practice.    

\subsubsection{Mini-batch backtracking algorithms} While being automatic and stable behaviour, Backtracking GD is still not popular compared to Standard GD, despite the latter requires a lot of efforts for manually fine-tuning. 

One reason is that when using Backtracking GD for a mini-batch, we obtain a learning rate which is good for this mini-batch but not for the full batch, as mentioned in the previous Subsection. We have resolved this by using rescaling of learning rates.  

Another reason is that Backtracking GD, which needs more computations at each step, is much slower than Standard GD. We can solve this problem by using Two-way Backtracking in Section 2. 

Combining the above two ideas, we arrive at a new method which we call Mini-batch Two-way Backtracking GD (MBT-GD), if we regard Backtracking GD as a learning rate finder for Standard GD. The precise procedure is as follows. We only need to apply Backtracking GD a small number of times (tens) at the beginning of the training process or periodically at the beginning of each epoch, for several first mini-batches, take the mean value of these obtained learning rates and rescale to achieve a good learning rate, and then switch to Standard GD. The cost for doing this is very small or negligible in the whole optimising process. We can also apply the same idea to obtain Mini-batch Two-way Backtracking MMT (MBT-MMT) and Mini-batch Two-way Backtracking NAG (MBT-NAG). 

The above idea is compatible with recent research (\cite{Wilson}, \cite{Zhang}), which suggests that traditional GD as well as MMT and NAG with a good learning rate can be on par or even better than Adam and other adaptive methods in terms of convergence to better local minima. The common practice in Deep Learning to find good learning rates is by manually fine-tuning, as mentioned above. Our Mini-batch Two-way Backtracking GD can be used to automatically fine-tune learning rates. Below we describe briefly the details of our mini-batch methods MBT-GD, MBT-MMT and MBT-NAG.

For MBT-GD, we simply compute an optimal learning rate at the beginning of the training process by applying backtracking line search to several mini-batches (e.g $20-50$ mini-batches), take the mean value of obtained learning rates with rescaling justification and use it in Standard GD for next training iterations. This method is inspired from cyclic learning rates (\cite{Smith}) and fastai learning rate finder (\cite{fastai},\cite{lrfinder}) (both require manual interference for the learning rate schedule), but use instead backtracking line search method in automatic manner (and hence is automatic). The recommended value for the hyper parameter $\alpha$ is $10^{-4}$, which means that we accept most of descent points but use scaling justification to reduce noise effects. This will give larger workable learning rates to accelerate speed and improve regularization (\cite{Smith2}). We keep using this learning rate value for Standard GD until the training gets stuck (e.g. no loss decreasing in $5$ consecutive epochs), at which time we then switch to Backtracking GD as in the beginning but now with $\alpha=0.5$. Using larger $\alpha$ when being near local minima gives us learning rates of appropriate size to guarantee convergence (otherwise, if learning rate is too big, then we can overshoot and leave the critical point, see Example \ref{Example1}).

For MBT-MMT and MBT-NAG, we use the same method to compute the optimal learning rate, but now at the beginning of every epoch (or after some fixed iterations of Standard GD) in order to take advantage of momentum accumulation which is the strong point of MMT and NAG. As in the case of MBT-GD, it is also recommended to use $\alpha=10^{-4}$ in most of the training process until the training gets stuck (e.g. no loss decreasing in $5$ consecutive epochs), at which time we then switch to use $\alpha=0.5$ and turn off momentum, too. This can help accelerate the convergence at the early stages of the training and take more care about the descent near local minima at later stages of the training. This is very similar to the methods of learning rate warming up and decay, see an illustration in Subsection \ref{Experiment1}. We also note that when doing experiments, we do not really follow the precise definition of Backtracking MMT and Backtracking NAG as in Section 2, which is complicated and which we will explore in more detail in future work. For the experiments here, we use the following simplified algorithm: fix the value of $\gamma$ to $0.9$ (as commonly used in practice) and choose $\delta _n$ by Backtracking GD. That is, we seek to find good learning rates in the standard MMT and NAG algorithms by using Backtracking GD. The intuition is that as Backtracking GD can find good learning rates for Standard GD, it can also find good learning rates for MMT and NAG. The experiments, see below, verify this speculation.       

For the hyper-parameter $\beta$, we could use any value from $0.5$ to $0.95$. By the very nature of backtracking line search, it is intuitive to see that a specific choice of $\beta $ does not affect too much the behaviour of Backtracking GD, see an illustration in Subsection \ref{Experiment2}. For the sake of speed, we use $\beta = 0.5$ and the number of mini-batches to apply backtracking line search (at the beginning of training or each epoch) to be $20$. This can help the training speed of MBT-MMT and MBT-NAG to be about $80\%-99\%$ of the training speed of MMT and NAG, depending on the batch size. The trade-off is inexpensive since we do not need to do manually fine-tuning of learning rate (which takes a lot of time and effort). 

It is worthy to note that the above settings for MBT-GD, MBT-MMT and MBT-NAG are  fixed in all experiments in this section and hence all obtained results come from entirely automatic training without any human intervention. 

\begin{remark}
Nowadays, in most of modern DNN, there are dropout and batch normalizations layers designed to prevent overfitting and to reduce internal covariate shift (\cite{Ioffe2015}). When using Backtracking GD for these layers, more cares are needed than usual, since non-deterministic outcomes of weights can cause unstable for our learning rate fine-tuning method. After many experiments, we find that some specific procedures can help to reduce the undesired effects of these non-deterministic outcomes. For dropout layers, we should turn them off when using Backtracking GD and turn them on again when we switch to Standard GD. For batch normalizations, we need to make sure using the training/testing flags in the consistent way to obtain the right values for Condition (\ref{Equation2}) and avoid causing non-deterministic and unstable for the backtracking process. 
\label{RemarkDropoutBatchNormalization}
\end{remark}

\subsection{Experiment 1: Behaviour of learning rates for Full Batch}\label{Experiment1}

In this experiment we check the heuristic argument in Subsection \ref{SubsectionHeuristics} for a single cost function. We do experiments with Two-way Backtracking GD for two cost functions: one is the Mexican hat in Example \ref{Example8}, and the other is the cost function coming from applying Resnet18 on a random set of 500 samples of CIFAR10. See Figure \ref{fig:lr_attenuation}.

\begin{figure}
\centering
        \begin{subfigure}[b]{0.5\textwidth}
                \includegraphics[width=\linewidth]{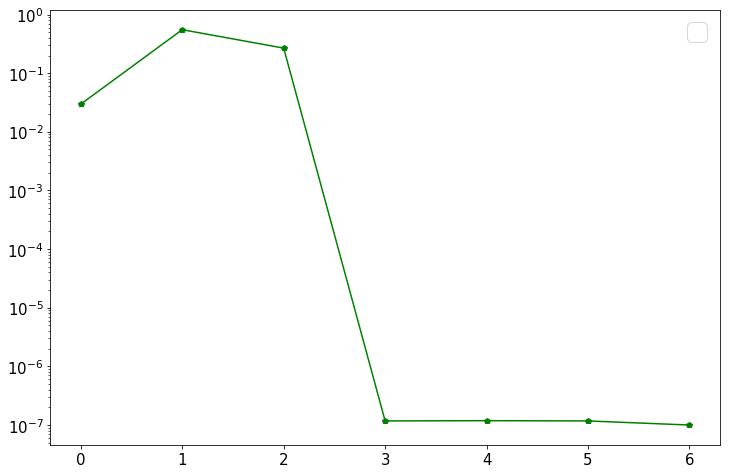}
                \caption{Mexican Hat}
                \label{fig:lr_MexicanHat}
        \end{subfigure}%
        \begin{subfigure}[b]{0.5\textwidth}
                \includegraphics[width=\linewidth]{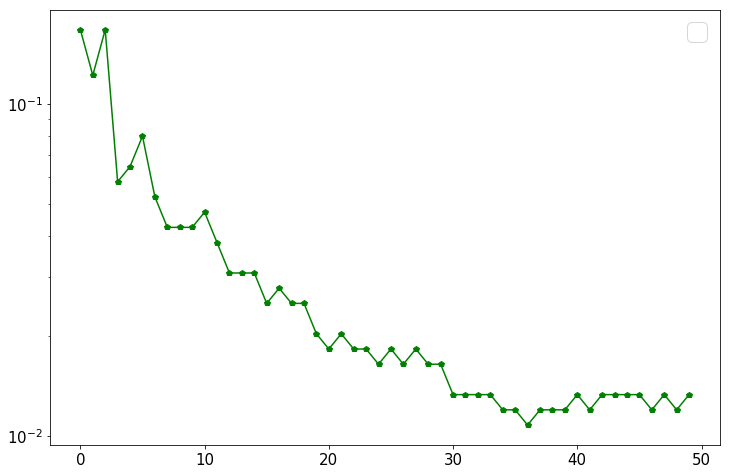}
                \caption{Resnet18}
                \label{fig:Resnet18}
        \end{subfigure}%
        \caption{Learning rate attenuation using Two-way Backtracking GD on (a) Mexican hat function (b) Resnet18 on a dataset contains 500 samples of CIFAR10 (full batch)}\label{fig:lr_attenuation}
\end{figure}

\subsection{Experiment 2: Behaviour of learning rates for Mini-Batch}\label{Experiment5}
In this experiment we check the heuristic argument in Subsection \ref{SubsectionHeuristics} in the mini-batch setting. We do experiments with MBT-MMT and MBT-NAG for the model Resnet18 on CIFAR10 and CIFAR100. See Figure \ref{fig:lr_attenuation_mini}.

\begin{figure}
\centering
        \begin{subfigure}[b]{0.5\textwidth}
                \includegraphics[width=\linewidth]{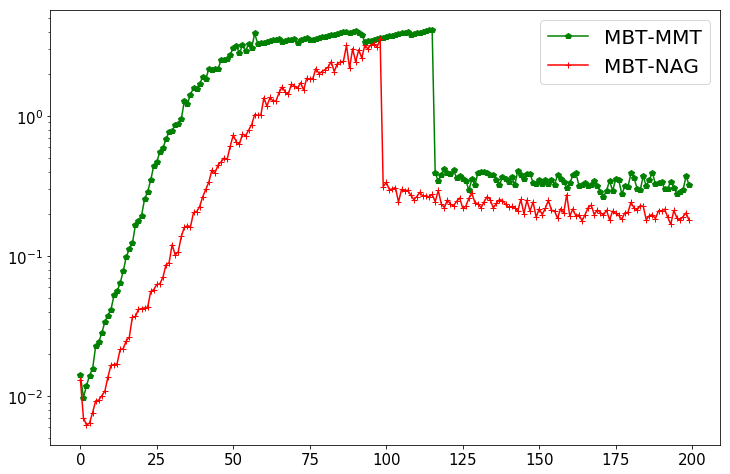}
                \caption{CIFAR10}
                \label{fig:lr_cifar10}
        \end{subfigure}%
        \begin{subfigure}[b]{0.5\textwidth}
                \includegraphics[width=\linewidth]{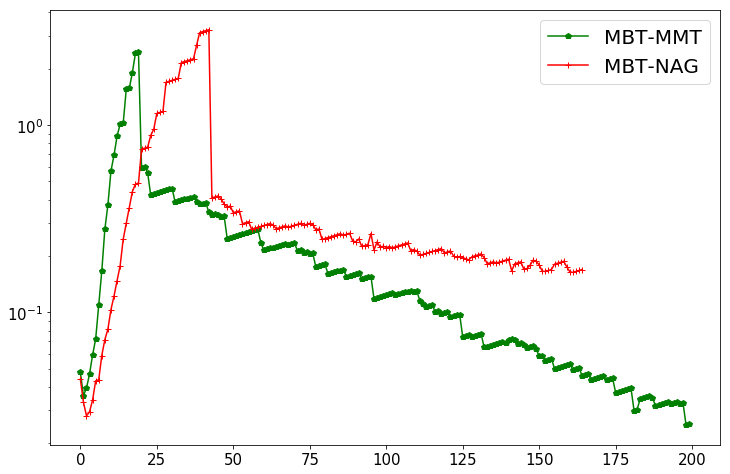}
                \caption{CIFAR100}
                \label{fig:lr_cifar100}
        \end{subfigure}%
        \caption{Learning rate attenuation using Two-way Backtracking GD in the mini-batch setting on Resnet18 on (a) CIFAR10 (b) CIFAR100 }\label{fig:lr_attenuation_mini}
\end{figure}

\subsection{Experiment 3: Stability of learning rate finding using backtracking line search}\label{Experiment2}

In this experiment, we apply MBT-GD to the network Resnet18 on the dataset CIFAR10, across $9$ different starting learning rates (from $10^{-6}$ to $100$) and $7$ different batch sizes (from $12$ to $800$), see Table \ref{tab:lr_batch}. With any batch size in the range, using the rough grid $\beta=0.5$ and only $20$ random mini-batches, despite the huge differences between starting learning rates, the obtained averagely-optimal learning rates stabilise into very close values. This demonstrates that our method works robustly to find a good learning rate representing the whole training data with any applied batch size. 

\begin{table}
\fontsize{8}{6}\selectfont
\setlength\extrarowheight{3pt}
  \centering
  \begin{tabular}{|r|c|c|c|c|c|c|c|c|c|c|}
  \hline
    LR				& $100$   & $10$    & $1$     & $10^{-1}$&$10^{-2}$ &$10^{-3}$& $10^{-4}$&$10^{-5}$	&$10^{-6}$\\
\hline
 $12$ & $0.0035$ & $0.0037$ & $0.0037$ & $0.0040$ & $0.0046$ & $0.0040$ & $0.0038$ & $0.0036$ & $0.0038$\\
\hline
$25$ & $0.0050$ & $0.0051$ & $0.0051$ & $0.0058$ & $0.0049$ & $0.0048$ & $0.0052$ & $0.0057$ & $0.0044$\\
\hline
$50$ & $0.0067$ & $0.0067$ & $0.0065$ & $0.0060$ & $0.0067$ & $0.0061$ & $0.0066$ & $0.0070$ & $0.0075$\\
\hline
$100$ & $0.0111$ & $0.0101$ & $0.0093$ & $0.0104$ & $0.0095$ & $0.0098$ & $0.0098$ & $0.0099$ & $0.0085$\\
\hline
$200$ & $0.0140$ & $0.0143$ & $0.0137$ & $0.0147$ & $0.0125$ & $0.0130$ & $0.0135$ & $0.0122$ & $0.0126$\\
\hline
$400$ & $0.0159$ & $0.0155$ & $0.0167$ & $0.0153$ & $0.0143$ & $0.0174$ & $0.0164$ & $0.0166$ & $0.0154$\\
\hline
$800$ & $0.0153$ & $0.0161$ & $0.0181$ & $0.0188$ & $0.0170$ & $0.0190$ & $0.0205$ & $0.0154$ & $0.0167$\\
\hline
 
\hline
  \end{tabular}
  \caption{Stability of Averagely-optimal learning rate obtained from the mini-batch two-way backtracking gradient descent across $9$ different starting learning rates (LR) ranging from $10^{-6}$ to $100$ and $7$ different batch sizes from $12$ to $800$. Applied using Resnet18 on CIFAR10. ($\alpha=10^{-4}$, $\beta=0.5$)} 
\label{tab:lr_batch}
\end{table}

\subsection{Experiment 4: Comparison of Optimisers}\label{Experiment3}
In this experiment we compare the performance of our methods (MBT-GD, MBT-MMT and MBT-NAG) with state-of-the-art methods. See Table \ref{tab:optimizers}. We note that MBT-MMT and MBT-NAG usually work much better than MBT-GD, the explanation may be that MMT and NAG escape bad local minima better. Since the performance of both MBT-MMT and MBT-NAG are at least $1.4\%$ above the best performance of state-of-the-art methods (in this case achieved by Adam and Adamax), it can be said that our methods are better than state-of-the-art methods.  

\begin{table}[htp]
\fontsize{8}{6}\selectfont
\setlength\extrarowheight{3pt}
  \centering
  \begin{tabular}{|l|c|c|c|c|c|c|c|c|c|c|}
  \hline
LR	& $100$   & $10$    & $1$     & $10^{-1}$&$10^{-2}$ &$10^{-3}$& $10^{-4}$&$10^{-5}$	&$10^{-6}$\\
\hline
SGD & $10.00$ & $89.47$ & $91.14$ & \it{92.07} & $89.83$ & $84.70$ & $54.41$ & $28.35$ & $10.00$\\
MMT & $10.00$ & $10.00$ & $10.00$ & \it{92.28} & $91.43$ & $90.21$ & $85.00$ & $54.12$ & $28.12$\\
NAG & $10.00$ & $10.00$ & $10.00$ & \it{92.41} & $91.74$ & $89.86$ & $85.03$ & $54.37$ & $28.04$\\
\hline
Adagrad & $10.01$ & $81.48$ & $90.61$ & $88.68$ & \it{91.66} & $86.72$ & $54.66$ & $28.64$ & $10.00$\\
Adadelta & $91.07$ & $92.05$ & \it{92.36} & $91.83$ & $87.59$ & $73.05$ & $46.46$ & $22.39$ & $10.00$\\
RMSprop & $10.19$ & $10.00$ & $10.22$ & $89.95$ & $91.12$ & \it{91.81} & $91.47$ & $85.19$ & $65.87$\\
Adam & $10.00$ & $10.00$ & $10.00$ & $90.69$ & $90.62$ & \it{92.29} & $91.33$ & $85.14$ & $66.26$\\
Adamax & $10.01$ & $10.01$ & $91.27$ & $91.81$ & \it{92.26} & $91.99$ & $89.23$ & $79.65$ & $55.48$\\
\hline
MBT-GD  & \multicolumn{9}{c|}{\it{91.64}}\\
MBT-MMT & \multicolumn{9}{c|}{\it{93.70}}\\
MBT-NAG & \multicolumn{9}{c|}{\bf{93.85}}\\
 
\hline
  \end{tabular}
   \caption{Best validation accuracy after $200$ training epochs (batch size $200$) of different optimisers using different starting learning rates (MBT methods which are stable with starting learning rate only use starting learning rate $10^{-2}$ as default). \it{Italic}: Best accuracy of the optimiser in each row.  \textbf{Bold}: Best accuracy of all optimisers for all starting learning rates.}  
  \label{tab:optimizers}
\end{table}

\subsection{Experiment 5: Performance on different datasets and models and optimisers}\label{Experiment4}
In this experiment, we compare the performance of our new methods on different datasets and models. We see from Table \ref{tab:models} that our automatic methods work robustly with high accuracies across many different architectures, from light weight model such as MobileNetV2 to complicated architecture as DenseNet121.

\begin{table}
  \centering
  \begin{tabular}{|l|c|c||c|c|}
  \hline
    Dataset &\multicolumn{2}{c||}{CIFAR10}	& \multicolumn{2}{c|}{CIFAR100}\\
    \hline
    Optimiser			& MBT-MMT   & MBT-NAG  & MBT-MMT   & MBT-NAG\\
    \hline
    Resnet18		& $93.70$   & $93.85$   & $68.82$  & $70.66$\\
    PreActResnet18	& $93.51$	& $93.51$	& $71.98$  & $71.53$ \\
    MobileNetV2		& $93.68$   & $91.78$   & $69.89$  & $70.33$\\
    SENet 		    & $93.15$	& $93.64$	& $69.62$  & $70.61$\\
    DenseNet121	    & \bf{94.67}& $94.54$   & $73.29$  &\bf{74.51}\\
    \hline
  \end{tabular}
  \caption{(Models and Datasets) Accuracy of state-of-the-art models after $200$ training epochs. \textbf{Bold}: Best accuracy on each dataset.} \label{tab:models}
\end{table}

\section{Conclusions}
 In this paper we showed that Backtracking GD method works very well for general $C^1$ functions. In particular, if $f$ has at most countably many critical points (for example if $f$ is a Morse function) then the sequence constructed from Backtracking GD either diverges to infinity or converges to a critical point of $f$. Some modifications, including an inexact version, are also available. The inexact version is then used to propose backtracking versions for popular methods MMT and NAG, where convergence is now proven under assumptions more general than for the standard versions of them, see Subsection \ref{SectionBacktrackingMMTNAG}. We also proved another result showing that in a certain sense it is very rare for any cluster point of the sequence to be a saddle point. We presented many examples illustrating various aspects of these results. Our method can also be applied to Wolfe's conditions, see Subsection \ref{SectionComparison}. 
  
 We then provided a heuristic argument showing that in the long run, Backtracking GD should stabilise to a finite union of Standard GD processes. Based on this, we proposed several modifications (Two-way Backtracking GD, Mini-batch Backtracking GD, MBT-GD, MBT-Momentum and MBT-NAG) which helps to save time in realistic applications in DNN. These modifications provide a very good automatic fine-tuning of learning rates. In fact, experiments with the MNIST and CIFAR10 data sets show that the new methods are better current state-of-the-art methods such as MMT, NAG, Adagrad, Adadelta, RMSProp, Adam and Adamax. 

 Theoretically, the above heuristic argument also suggests that good properties of the Standard GD still hold for the Backtracking GD. In particular, we expect that the following conjecture is true concerning saddle points.   
 
\begin{conjecture}
Assume that $f:\mathbb{R}^m\rightarrow \mathbb{R}$ is a $C^1$ function and is $C^2$ near its critical points. Then the set of initial points $z_0\in \mathbb{R}^m$ for which the cluster points of the sequence $\{z_n\}$ in (\ref{Equation3}) contain a saddle point has Lebesgue measure $0$. 
 \label{Conjecture1}\end{conjecture}
 
 We note that this statement is stronger than the known results for the Standard GD. This is due to the fact that we have part 2) in Theorem \ref{Theorem1} for the Backtracking GD, which is unavailable for the Standard GD. 
 
Theorem \ref{Theorem2} supports this conjecture. Besides this theorem and the above heuristic argument, here is a more realistic approach and evidence towards Conjecture \ref{Conjecture1}. Assume for simplicity that  $f$ has at most countably many critical points. If the cluster points of the sequence $\{z_n\} $ contains a saddle point, then by Theorem \ref{Theorem1} the whole sequence $\{z_n\}$ converges to $z_{\infty}$. As we showed previously in this paper, since $f$ is $C^2$ around $z_{\infty}$, the sequence $\delta (f,\delta _0,z_n)$ will then take values in a fixed finite set. Thus, the Backtracking GD then should be at most a combination of a finite number of sequences constructed from the standard gradient descent method. Then we expect that results and arguments from \cite{lee-simchowitz-jordan-recht, panageas-piliouras} apply.  
 
 In the same vein, we state some other general open questions of great interest, which we hope to address in the near future. 
 
{\bf Question 1.} Can Theorem \ref{Theorem1} be extended to all $C^1$ functions or to functions on infinite dimensions $f:\mathbb{R}^{\infty}\rightarrow \mathbb{R}$? This question can have applications to other fields such as PDE, where to solve a PDE we can try to reduce to the problem of finding extremal points of functionals on a space of functions, which is of infinite dimension. 

{\bf Question 2.} Can we prove a version of Theorem \ref{Theorem1} for constrained optimisation, where the variables $x$ are constrained in a convex subset $D$ of $\mathbb{R}^m$? In convex optimisation, there is a method called projected gradient descent to deal with this problem, and we speculate that a similar procedure may work for Question 2.  

{\bf Question 3.} Is there a stochastic version of Theorem \ref{Theorem1}, as Stochastic GD is for Standard GD? Note that in the current version of Stochastic GD, Armijo's condition is not considered, and that when considering Backtracking GD, in general the learning rates $\delta (f,\delta _0,x)$ are not continuous as a function in $x$ as shown in Example \ref{Example2}.

{\bf Question 4.} What are the roles of the hyper-parameters $\alpha, \beta ,\delta _0$ in Backtracking GD? This question is a simple case of how to deal with hyper-parameters in optimisation, and itself has important implications for Deep Learning. As a first step, we will find ways to formulate this as an optimisation problem itself. The same question can be asked for MMT and NAG, as well as their backtracking versions as proposed in Section 2. 

\bibliographystyle{siamplain}

\bibliography{references}
\end{document}